\documentclass[12pt,letterpaper]{article}

\usepackage[utf8x]{inputenx}
\usepackage[T1]{fontenc}
\usepackage{textcomp}
\usepackage{lmodern}
\usepackage{soul}
\usepackage{verbatim}

\usepackage{amsthm,amsmath,amssymb}
\usepackage{amsfonts}
\usepackage{mathcomp}
\usepackage{mathrsfs}
\usepackage{euscript}

\usepackage{graphics}
\usepackage{graphicx}
\usepackage{caption}
\usepackage{booktabs}
\usepackage{floatrow} 
\usepackage{subfig}
\usepackage{psfrag}
\usepackage{epic,eepic}
\usepackage{color}
\usepackage[table]{xcolor}
\usepackage{tikz}

\usepackage{cite}
\usepackage{enumerate}
\usepackage{paralist}
\usepackage{amsrefs}   

\usepackage{fullpage}

\usepackage[left, mathlines]{lineno}
\usepackage{float}

\theoremstyle{plain}
\newtheorem{thm}{Theorem}
\newtheorem{lem}[thm]{Lemma}
\newtheorem{cor}[thm]{Corollary}
\newtheorem{prop}[thm]{Proposition}

\newtheorem{obs}[thm]{Observation}

\theoremstyle{definition}

\newtheorem{conj}[thm]{Conjecture}

\theoremstyle{remark}

\def\Z{{\mathbb{Z}}}

\begin{document}

\title{The packing number of the double vertex graph of the path graph}

\author{J. M. G\'omez Soto \and J. Lea\~nos \and L. M. R\'ios-Castro \and L. M. Rivera}
\date{}

\maketitle{}

\begin{abstract}
Neil Sloane showed that the problem of determining the maximum size of a binary code of constant 
weight $2$ that can correct a single adjacent transposition is equivalent to finding the packing number of a certain graph.
 In this paper we solve this open problem by finding the packing number of the double vertex graph ($2$-token graph) of a path graph. 
This double vertex graph is isomorphic to Sloane's graph. 
Our solution implies a conjecture of Rob Pratt about the ordinary generating function of sequence A085680. 
\end{abstract}

{\it Keybwords:}  Error-correcting-codes, Binary codes; Token graphs; Packing number. 

\section{Introduction}

Let $n$ be a positive integer and let $w\in \{0,1,\ldots ,n \}$.  A {\em binary code of length $n$ and weight $w$} is a subset $S$ of $\mathbb{F}_2^n$ such that every element in $S$ has exactly $w$ 1's and $n-w$ 0's. The elements of $S$ are called {\em codewords}. For a binary vector $u \in \mathbb{F}_2^n$, let $B_e(u)$ denote the set of all binary vectors 
in $\mathbb{F}_2^n$ that can be obtained from $u$ as a consequence of certain error $e$. For example, $e$ can be the deletion or the transposition of bits. In fact, in 
J. A. Gallian  \cite{galli} we can see  that for the case of digit codes and certain channels, experimental results show that these two are precisely the more common errors. A subset $C$ of $\mathbb{F}_2^n$ is said to be an {\em $e$-deletion-correcting code} if $B_e(u)\cap B_e(v)=\emptyset$, for all  $u, v \in C$ with $u \neq v$. A classical problem in coding theory is to find the largest $e$-deletion-correcting code. In the survey of Sloane~\cite{sloane1} we can see results about single-deletion-correcting-codes and in Butenko et al. \cite{buten}, and in the web page \cite{sloane2} some results about transposition error correcting codes. 

In this paper we consider the case when error $e$ consists of a single adjacent transposition (adjacent bits may be swapped). The sequence A057608 in OEIS~\cite{oeis} corresponds  precisely to the maximal size of a binary code of length $n$ that corrects one single adjacent transposition, and the range of values of $n$ for which A057608($n$) is known is from $0$ to $11$ (\cites{buten,oeis}). More information about this sequence and its equivalent formulation in the setting of graphs can be found in the web page~\cite{sloane2}. When we restrict our attention to codes of length $n$ and constant weight $k$, then the corresponding sequence in OEIS~\cite{oeis} is A085684, and  corresponds to the size $T(n,k)$ of the largest code of length $n$ and constant weight $k$ that can correct a single adjacent transposition. The case in which $k=2$ has been studied at least since 2003 and corresponds to the sequence A085680 in OEIS~\cite{oeis}.  The range of values of $n$ for which A085680($n$) is known is from $2$ to $50$.   

The problem of determining A085680($n$) was formulated in the setting of graph theory by Sloane~\cites{sloanep} as follows: let $\Gamma_n$ be the graph whose vertices are all
 binary vectors of length $n$ and constant weight $2$, and two vertices are adjacent if and only if one can be obtained from the other by transposing a pair of adjacent bits. 
 From the definition of $\Gamma_n$ and the requirement that $B_e(u)\cap B_e(v)=\emptyset$, for any distinct vertices $u$ and $v$ of $\Gamma_n$, it follows that any binary code $C$ with constant weight $2$ can correct a single adjacent transposition if and only if $C$ is a packing set of $\Gamma_n$. Therefore, the maximum cardinality of such a code $C$  is equal to the packing number $\rho(\Gamma_n)$ of $\Gamma_n$. In other words, $A085680(n)=\rho(\Gamma_n)$. 
 
 Let $G$ be a graph. We recall that a set $S\subseteq V(G)$ is a \emph{packing set} of $G$ if every pair of distinct vertices $u, v\in S$ satisfy $d_G(u,v)\geq 3$, where
$d_G(u,v)$ denotes the distance between $u$ and $v$ in $G$. The {\em packing number} $\rho(G)$ of $G$ is the maximum cardinality of a packing set of $G$. 
Note that the notion of  packing set can be extended to infinite graphs in a natural way.  

 Surprisingly, the graph $\Gamma_n$ is isomorphic to the $2$-token graph $F_2(P_n)$ of the path graph $P_{n}$.  Since the structure of $\Gamma_n$ can be more easily explained in the context of the $k$-token graphs, let us give an overview  on the $k$-token graphs in the next subsection.

\subsection{Token graphs}
Let $G$ be a graph of order $n$ and let $k$ be an integer such that $1 \leq k \leq n-1$. The $k$-token graph $F_k(G)$ of  $G$ is defined as the graph with vertex set all $k$-subsets of $V(G)$, where two vertices are adjacent in $F_k(G)$ whenever their symmetric difference is an edge of 
$G$. Note that $F_k(G)$ is isomorphic to $F_{n-k}(G)$. In Figure \ref{fig:fichasdeP6} 
we show the $2$-token graph of $P_5$ and the graph $\Gamma_5$.

\begin{figure}[ht]
\begin{center}
\includegraphics[width=.75\textwidth]{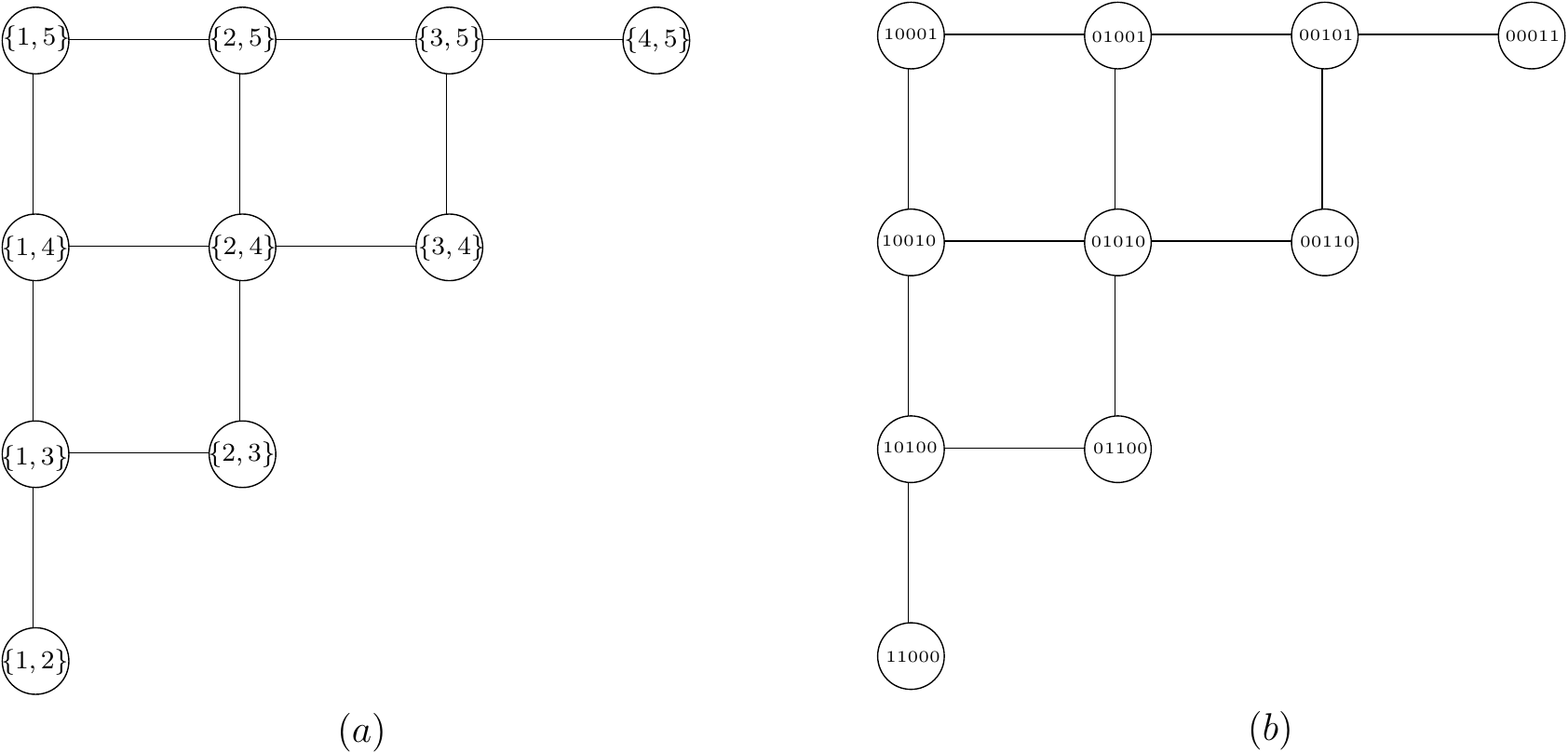}
\caption{\small (a) The $2$-token graph of $P_5$. (b) Graph $\Gamma_5$.}
\label{fig:fichasdeP6}
\end{center}
\end{figure}
 
 The $k$-token graphs were defined by Zhu, Liu, Dick and Alavi in \cite{zhu} and are a generalization of the double vertex graphs ($k=2$) defined by Alavi, Behzad and Simpson in \cite{alavi1}. Several combinatorial properties for double vertex graphs were studied by different authors (see, eg., \cites{alavi1,alavi2,alavi3}).

Later, T. Rudolph \cites{rudolph, aude}  redefined the token graphs, with the name of symmetric powers of graphs, in connection with quantum mechanics and the graph isomorphism problem. Rudolph presented examples of cospectral non-isomorphic graphs $G$ and $H$, such that the corresponding $2$-token graphs are non-cospectral. Then, as was noted by Rudolph, the eigenvalues of the $k$-token graphs seems to be a more powerful invariant than those of the original matrix $G$. However Barghi and Ponomarenko \cite{barghi} and,  independently, Alzaga et al. \cite{alzaga}, proved that for a given positive integer $k$ there exist infinitely many pairs of non-isomorphic graphs with cospectral $k$-token graphs. In 2012 Fabila-Monroy et al. \cite{FFHH} reintroduced the concept of $k$-token graph as another type reconfiguration problem \cites{caline,demaine,vanden, yama} and began the systematic study of the combinatorial parameters of $F_k(G)$. This approach has been followed in \cites{token2,token3,leatrujillo, rive-tru}, where problems related to connectivity, diameter, clique number, chromatic number, independence number, Hamiltonian paths, matching number, planarity, regularity, etc., of token graphs are studied.

Note that token graphs are also a generalization of Johnson graphs: if 
$G$ is the complete graph of order $n$, then $F_k(G)\simeq J(n,k)$. The Johnson graph was introduced by Johnson~\cite{Jo} in connection with codes, and have been widely studied from several points of view, see for example \cites{etzion,terwilliger}.

We end this subsection by noting that $\Gamma_n\simeq F_2(P_n)$.  Suppose that $V(P_n)=\{1,\ldots ,n\}$ and that $E(P_n)=\{\{i, i+1\} \colon~ 1 \leq i \leq n-1\}$. 
Consider $f\colon V(F_2(P_n))\to V(\Gamma_n)$ defined as follows: for $A \in V(F_2(P_n))$, let $f(A):=(a_1,\dots ,a_n)$, where $a_i=1$ if $i \in A$ and $a_i=0$, otherwise.
It is easy to check that such an $f$ is an isomorphism.  Hence A085680($n$)=$\rho(F_2(P_n))$.

\subsection{Main result}

The exact value for A085680($n$) before march of 2017 was known only for $n\in \{2,\ldots ,25\}$. Later, Rob  Pratt reported in \cite{oeis} the exact values of A085680($n$) for 
$n\in \{26,\ldots ,50\}$, and posed the following conjecture about the ordinary generating function for this sequence. 

\begin{conj}[Rob Pratt]\label{conjecture} 
$\sum_{n\geq 0} A085680(n+2)x^n=\frac{1-x+x^2-x^{10}+x^{11}}{(1-x)^2(1-x^5)}$.
\end{conj}

Consider the following sequence:  

\begin{equation*}\label{a(n)}
a(n) := \left\{ \begin{array}{ll}
             \frac{1}{10}\left(n^2+n+20\right)  & \text{ if } n \equiv 0 \mbox{ (mod 5)  or } n\equiv 4 \mbox{ (mod 5)},\\ 
             \frac{1}{10}\left(n^2+n+18\right)    & \text{ if } n \equiv1 \mbox{ (mod 5)  or } n\equiv 3 \mbox{ (mod 5)},\\
             \frac{1}{10}\left(n^2+n+14\right)  & \text{ if }  n \equiv 2  \mbox{ (mod 5)}.\\
            \end{array}
   \right.
\end{equation*}

Unless otherwise stated, for the rest of the paper, $a(n)$ is as above. The following remark follows immediately from the definition of $a(n)$. 

\begin{obs}\label{formula:rec}
For any integer $n\geq 6$ 
\[
a(n)=a(n-5)+n-2,
\]

where the first five values are $a(1)=2, a(2)=2, a(3)=3, a(4)=4$ and $a(5)=5$. 
\end{obs}

The aim of this paper is to show Conjecture \ref{conjecture} by giving an explicit formula for A085680($n+1$), namely $a(n)$.  
In Table~\ref{tablaan} we shown the first eighteen values of $A085680(n+1)$ and $a(n)$. 

\begin{table}[htp]
\caption{First eighteen values of A085680($n+1$) and $a(n)$}
\begin{center}
{\footnotesize
\begin{tabular}{|c|c|c|c|c|c|c|c|c|c|c|c|c|c|c|c|c|c|c|c|c|c|}
\hline
$n$ &1 &2&3&4&5&{\bf 6}&{\bf 7}&{\bf 8}&{\bf 9}&{\bf 10}&{\bf 11}&{\bf 12}&{\bf 13}&{\bf 14}&{\bf 15}&{\bf 16}&{\bf 17}&{\bf 18}\\
\hline
A085680($n+1$) & 1& 1& 2& 3& 4& 6& 7& 9& 11& 13& 15& 17& 20& 23& 26& 29& 32& 36\\
\hline
\footnotesize $a(n)$ &2& 2& 3& 4& 5& 6& 7& 9& 11& 13& 15& 17& 20& 23& 26& 29& 32& 36\\
\hline
\end{tabular}
}
\end{center}
\label{tablaan}
\end{table}%

Our main result is the following.

\begin{thm}\label{thm:main}
Let $n\geq 6$ be an integer. Then A085680($n+1$)$= a(n)$. 
\end{thm}

Due to the values in Table~\ref{tablaan} it is enough to consider $n >10$.  
For readability, we will separate the proof of Theorem \ref{thm:main} in the following lemmas.  

\begin{lem}\label{lemma:lower}
Let $n> 10$ be an integer. Then A085680($n+1$)$\geq a(n)$. 
\end{lem}

\begin{lem}\label{lemma:upper}
Let $n> 10$ be an integer. Then A085680($n+1$)$\leq a(n)$. 
\end{lem}

\subsection{Notation}\label{notacion}

We recall that the {\em infinite integer grid} $H$ is the infinite plane graph with vertex set $\Z^2=\Z\times \Z$ that has has an edge joining the vertices  $(i,j)$ and $(i',j')$ if and only if  $|i-i'|+|j-j'| = 1$. As usual let $\Z_5=\{0,1,2,3,4\}$ denote the group of integers modulo $5$.

The homomorphism $f \colon \Z^2\to \Z_5$ defined by $f(i, j)=(i+2j) \bmod {5}$ will play a central role in the next section. 

Let $T(n)$ be the subgraph of $H$ induced by the following vertex subset of $H$:   
\[
V(n)= \left\{(i,j):  1 \leq i,j\leq n \mbox{ and } i-j\leq 0\right\}.
\]

Note that $T(n)\simeq F_2(P_{n+1})\simeq \Gamma_{n+1}$. In view of this, for the rest of the paper we shall use $T(n)$ instead of $F_2(P_{n+1})$. 

 
For $i,j,k,r \in \{1,\ldots ,n\}$, with $i\leq j$ and $k< r$, let $T_{i,j}^{k,r}(n)$ be the subgraph of $T(n)$ induced by the following set of vertices: 
\[
\lbrace (x,y): i\leq x\leq j, k\leq y\leq r\rbrace.
\]
 For brevity, if $\ell$ is a positive integer such that  $n-\ell \geq 0$, we write $T_{i,j}^{\ell}(n)$ and $T_i^{\ell}(n)$ instead of $T_{i,j}^{n-\ell+1,n}(n)$ and $T_{i,i}^{n-\ell+1,n}(n)$, respectively. Moreover, if $i=1$ and $j=n$ we will use $T^{\ell}(n)$ and $T^{k,r}(n)$ instead of $T_{1,n}^{\ell}(n)$ and $T_{1, n}^{k,r}(n)$, respectively. If $S$ is a packing set of $T_{i,j}^{k,r}(n)$ we define $S_{i,j}^{k,r}(n)=T_{i,j}^{k,r}(n)\cap S$. In Figure \ref{fig:ejemplos} we shown $T(17)$ and  several such subgraphs.

\begin{figure}
\centering
  \begin{minipage}{0.53\textwidth}
    \centering
    \includegraphics[width=0.95\textwidth]{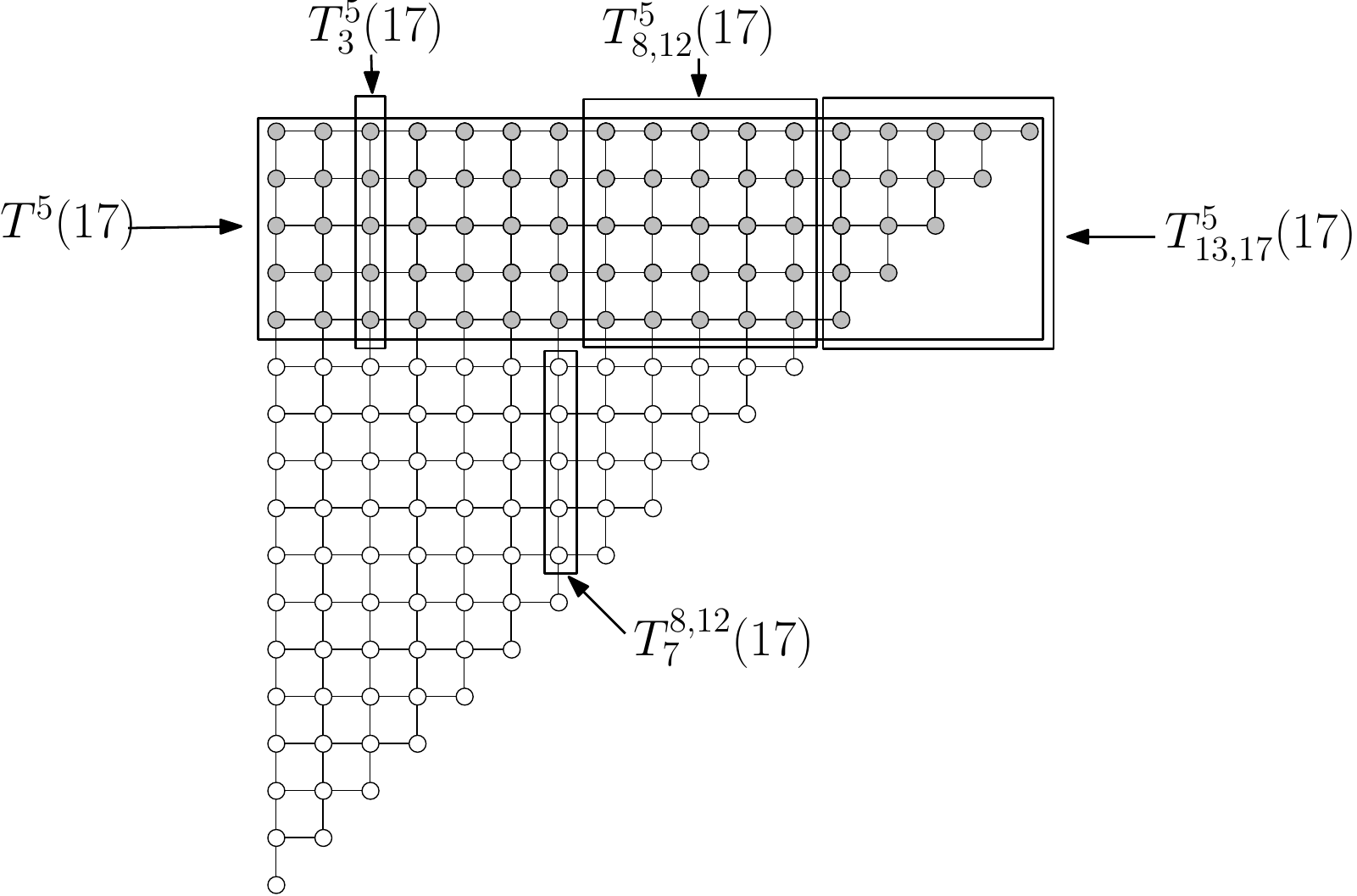}\\
    {\small (a)} 
  \end{minipage}
  \hspace{1mm}
  \begin{minipage}{0.43\textwidth}
    \centering
    \includegraphics[width=0.95\textwidth]{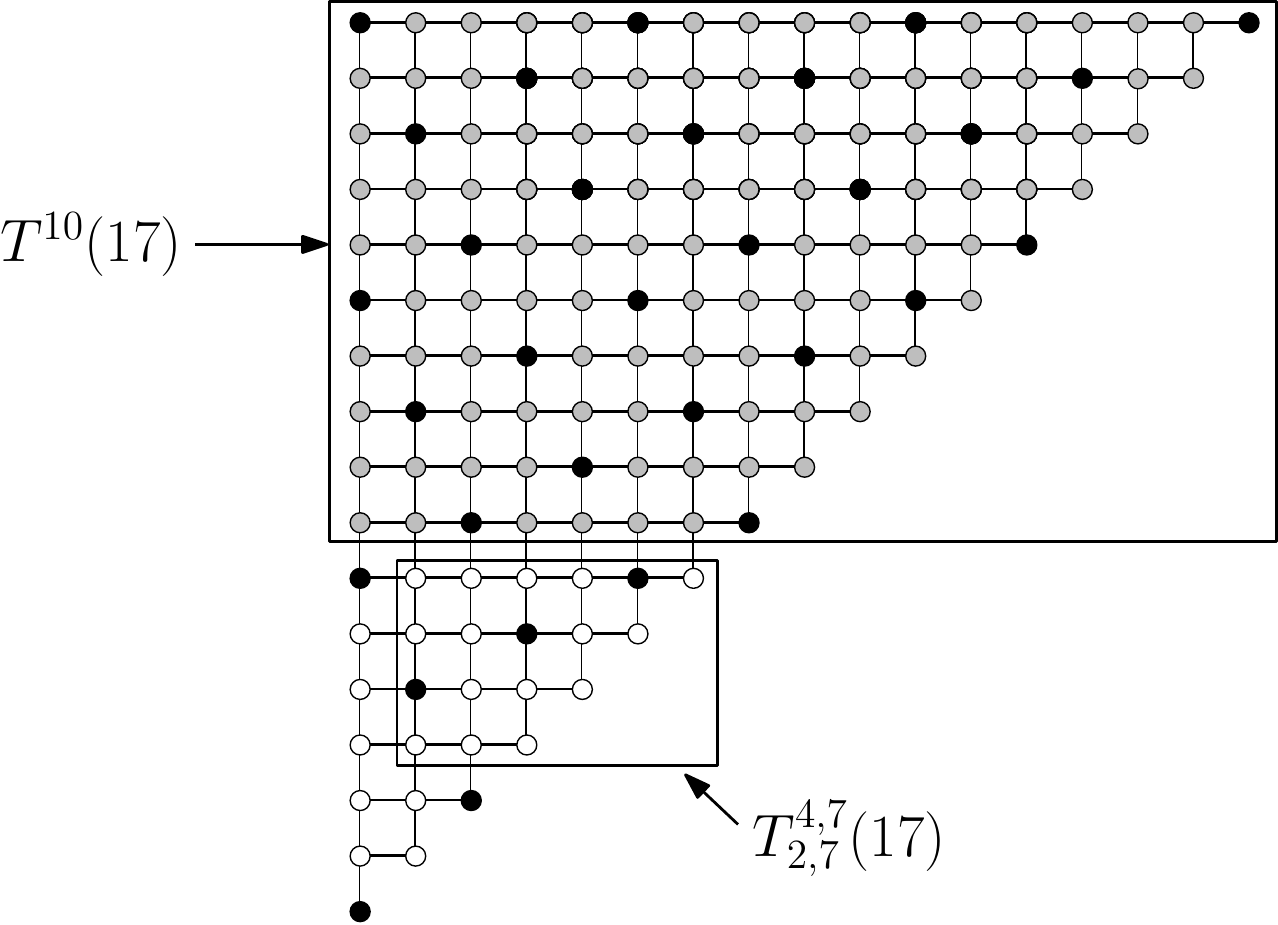}\\
    {\small (b)} 
  \end{minipage}
  \caption{\small Here we shown two copies of $T(17)$. (a) The subgraph of $T(17)$ induced by the gray vertices is $T^5(17)$. (b) The set of black vertices $S$ is a packing set 
  of $T(17)$, and those in $T^{10}(17)$ form $S^{10}(17)$.}
  \label{fig:ejemplos}
\end{figure}


\section{The sequence $\rho(T(n))$ is lower bounded by $a(n)$}

Our goal in this section is to show Lemma \ref{lemma:lower} by constructing packing sets for $T(n)$ with cardinality $a(n)$.

In this section we will use slight refinements of the preimages $f^{-1}(\ell)$, $\ell\in \Z_5$ ($f$ is the homomorphism defined in previous section) in order to get the constructions behind of our lower bound for $\rho(T(n))$. For example, in Figure \ref{fig:fichasdeP7-1} (left) we have $T(7)\cap f^{-1}(3)$.

We say that each of these five preimages of $f$ satisfies the ``horse pattern'' because  each such set $f^{-1}(\ell)$ can be produced by certain types of horse movements in the game of chess (with the vertex set of $H$ as the chess-board). 
As we shall see, the inverse image $f^{-1}(\ell)$ of each $\ell\in \Z_5$ under $f$ is a vertex subset of $H$ which is a packing set of $H$. In this section we exhibit some basic properties of such preimages. In particular, our main goal here is to show Proposition \ref{values-franja}, which determines the size of such preimages on $T(n)$. 
  
 For $t\in\Z^+$ and $(i,j)\in \Z^2$, we define $f_t \colon \Z^2\to \Z_5$ as follows: $f_t(i,j)=f(i-t,j)$. 
We call $f_t$ the {\em $t$-translation} of $f$.
 
In the following proposition we establish some basic facts involving $H, f$ and $f_t$. The proof of each item is straightforward from the  definitions or previous items. 

\begin{prop}\label{prop:H}
Let $\ell\in\Z_5$, $i,i',j,j'\in\Z$, and let $H$ and $f$ be as above. Then 
\begin{itemize}
\item[(i)] $f$ is a proper vertex $5$-coloring of $H$.
\item[(ii)]  $f^{-1}(\ell)$ is a packing set of $H$. 
\item[(iii)] If $G$ is a subgraph of $H$, then $G\cap f^{-1}(\ell)$ is a packing set of $G$. 
\item[(iv)]  $f(i,j) = f(i,j')$ if and only if $j-j'$ is a multiple of $5$. 
\item[(v)]  $f(i,j) = f(i',j)$ if and only if  $i-i'$ is a multiple of $5$. 
\item[(vi)] $f(i+1,i) = f(j+1,j)$ if and only if $i-j$ is a multiple of $5$. 
\item[(vii)] If $t\in \Z^+$, then $f^{-1}(f(i,j))=f_t^{-1}(f(i-t,j))$. 
\end{itemize} 
\end{prop}

We recall that a {\em $5$-coloring} of $H$ is a function $c\colon V(H) \to \Z_5$ such that if $u$ and $v$ are adjacent vertices in $H$, then $c(v)\neq c(u)$. 
Thus, $f$ is a proper $5$-coloring of $H$. We shall refer to the elements of $\Z_5$ as the colors of $f$ and for $\ell\in \Z_5$, we shall refer to $f^{-1}(\ell)$ as the $\ell$ chromatic class of $H$ under $f$.

We remark that Proposition \ref{prop:H} $(vii)$ establishes that if $\ell$ and $\ell'$ are colors in $\Z_5$, then the chromatic class $f^{-1}(\ell)$ 
is a homothetic copy of $f^{-1}(\ell')$.

 \begin{figure}[H]
\begin{center}
\includegraphics[width=.50\textwidth]{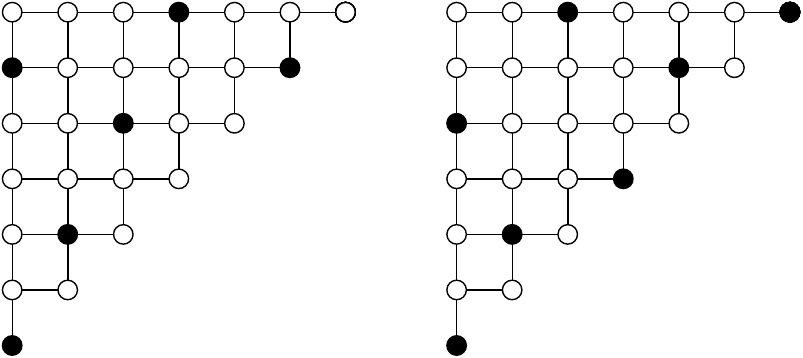}
\caption{\small Two copies of $T(7)$. The two sets of black vertices are packing sets of $T(7)$. The packing set on the left is $T(7)\cap f^{-1}(3)$.}
\label{fig:fichasdeP7-1}
\end{center}
\end{figure}
 
\begin{prop}\label{values-franja}
Let $i\in\{5,\dots, n\}$. If $m \in \Z_5$, then $|T^{i-4,i}(n)\cap f^{-1}(m)|=i-2$. 
\end{prop}

\begin{proof} 
First we show that $|T^{i-4,i}(n)\cap f^{-1}(m)|=|T^{i-4,i}(n)\cap f^{-1}(m')|$ for any $m,m'\in \Z_5$. From Proposition  \ref{prop:H} $(v)$ we know that the colors $c_1:=f(1,i)$, $c_2:=f(2,i)$,$\ldots $,$c_5:=f(5,i)$ are distinct. Thus, it is enough to show that  $|T^{i-4,i}(n)\cap f^{-1}(c_t)|=|T^{i-4,i}(n)\cap f^{-1}(c_{t+1})|$ for any $t\in \{1,\dots, 4\}$. Let 
$t\in \{1,\dots, 4\}$. To help comprehension, and without loss of the generality, let $c_t$ be the blue color and let $c_{t+1}$ be the red color. Let $b$ be the number of vertices of $T^{i-4,i}(n)$ which are colored blue by $f$. Formally, $b:=|T^{i-4,i}(n)\cap f^{-1}(c_t)|$. Note that the number $b_1$ of vertices of $T^{i-4,i}(n)$ which are colored blue by $f_1$ is $b+p-q$, where $p$ denotes the number of vertices in $\{(0,i-4),(0,i-3), (0,i-2),(0,i-1),(0,i)\}$ that are colored blue by $f$, and $q$ denotes the number of vertices in $\{(i-4,i-4),(i-3,i-3), (i-2,i-2),(i-1,i-1),(i,i)\}$ that are colored blue by $f$. From Proposition \ref{prop:H} $(iv)$ and $(vi)$ we know, respectively, that $p=1$ and $q=1$. Hence $b=b_1$. But by Proposition \ref{prop:H} $(vii)$ we know that the set of vertices of $T^{i-4,i}(n)$ that are colored red by $f$ is equal to the set of vertices of $T^{i-4,i}(n)$ that are colored blue by $f_1$. This and $b=b_1$ imply that $|T^{i-4,i}(n)\cap f^{-1}(c_t)|=|T^{i-4,i}(n)\cap f^{-1}(c_{t+1})|$, as desired.
 
As $T^{i-4,i}(n)$ has $(i-4)+(i-3)+(i-2)+(i-1)+i=5(i-2)$ vertices, and $|T^{i-4,i}(n)\cap f^{-1}(m)|=|T^{i-4,i}(n)\cap f^{-1}(m')|$ for any $m,m'\in \Z_5$, then $|T^{i-4,i}(n)\cap f^{-1}(m)|=i-2$ for any $m\in \Z_5$.
\end{proof}

  
\subsection{Proof of Lemma \ref{lemma:lower}}

We recall that $n$ is an integer greater to $10$. Depending on the value of 
$t:=n \bmod{5}$, we will give a packing set $A_{n,t}$ of $T(n)$ such that 
$|A_{n,t}| = a(n)$, as required.   

Consider the following vertex subsets of $T(n)$ (see Figure \ref{fig:construction}). 
\begin{itemize}

\item For $n\equiv 1 \mbox{ (mod 5)}$, we define
\[ 
A_{n,1}=(T(n)\cap f^{-1}(4))  \bigcup \{(1, 1), (n, n)\}.
\]

\item For $n\equiv 2 \mbox{ (mod 5)}$, we define
\[
A_{n,2}=\left((T(n)\cap f^{-1}(0))\setminus \{ (1, 2), (2, 4)\}\right) \bigcup \{(1, 1), (1, 4), (3, 3)\}.
\]

\item For $n\equiv 3 \mbox{ (mod 5)}$, we define
\[
A_{n,3}=\left((T(n)\cap f^{-1}(0)) \setminus \{ (1, 2), (2, 4)\}\right) \bigcup \{(1, 1), (1, 4), (3, 3), (n, n)\}.
\]

\item For $n\equiv 4 \mbox{ (mod 5)}$, we define
\[A_{n,4}=\left((T(n)\cap f^{-1}(0)) \setminus  \{(1, 2), (2, 4), (n-2, n), (n-3, n-2)\}\right) \bigcup\]
\[\{(1, 1), (1, 4), (3, 3), (n-3, n), (n-2, n-2), (n, n)\}.\]

\item For $n\equiv 0 \mbox{ (mod 5)}$, we define
\[ A_{n,0}= \left((T(n)\cap f^{-1}(1)) \setminus \{(1, 5), (2, 2), (2, 7), (3, 4), (4, 6)\}\right)\bigcup\]
\[  \{(1, 1), (1, 4), (1, 7), (3, 3), (3, 6), (5, 5), (n, n)\}.\]
\end{itemize}

\begin{figure}[h]
\begin{center}
\begin{tabular}{ccc}
\includegraphics[width=.250\textwidth]{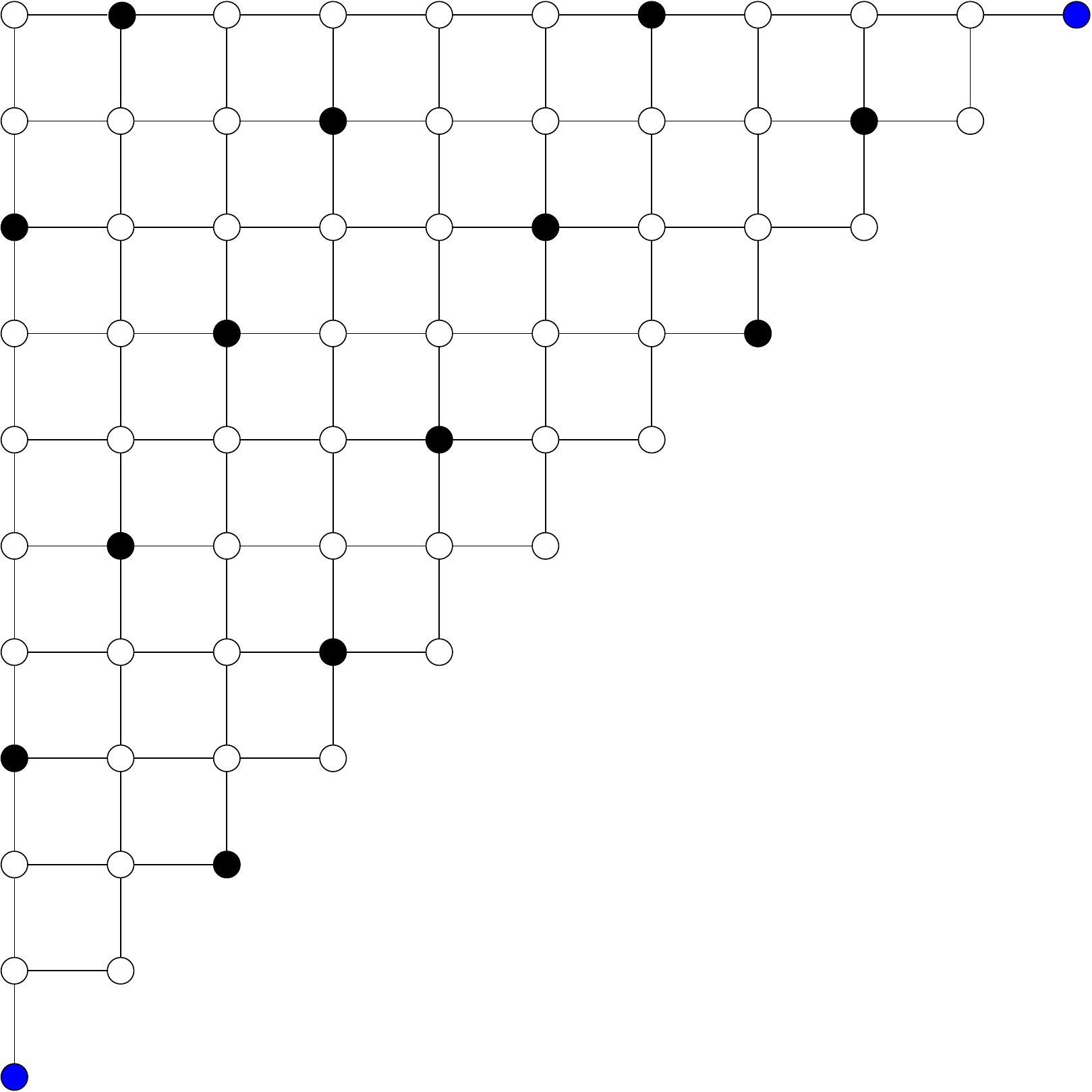}&\includegraphics[width=.250\textwidth]{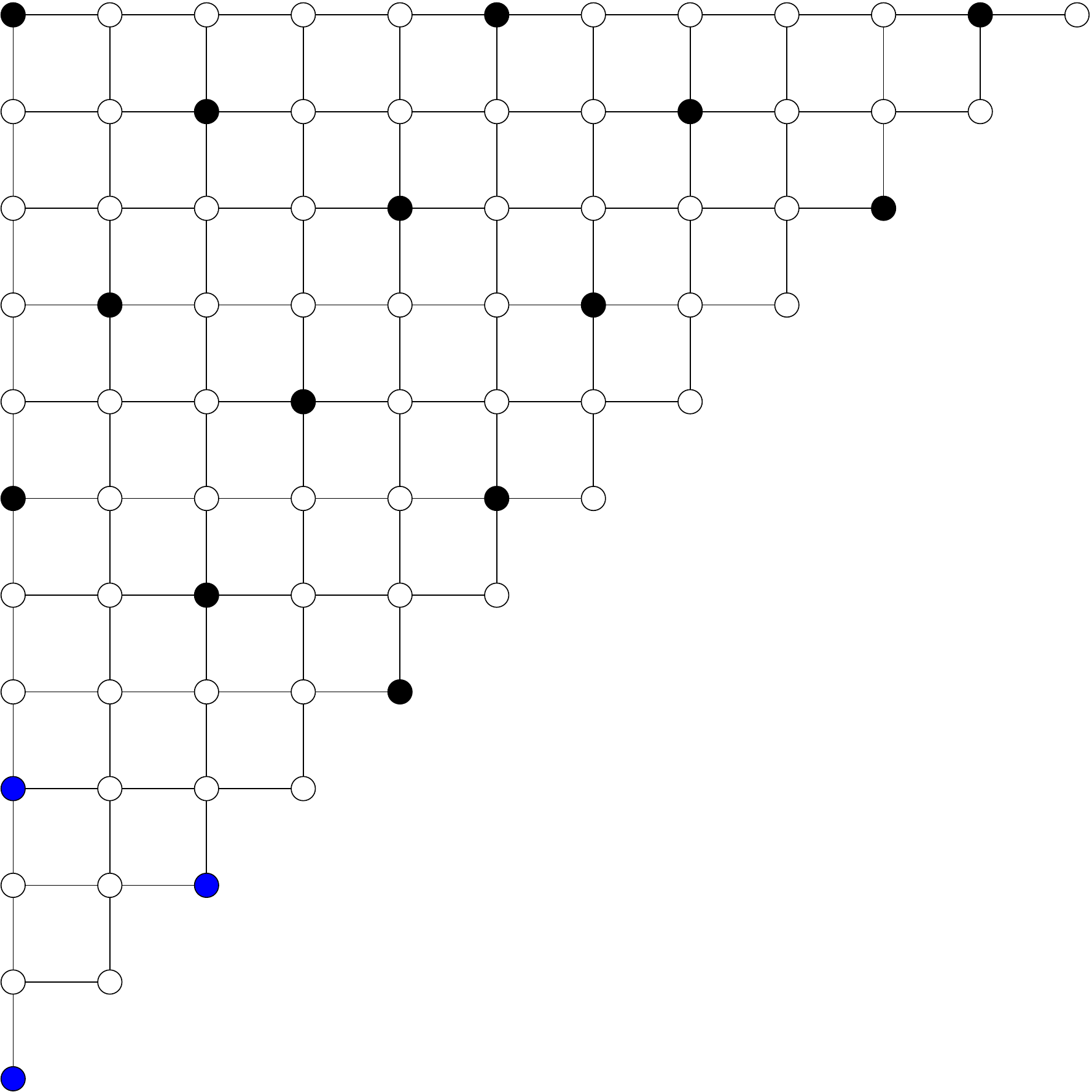}&\includegraphics[width=.250\textwidth]{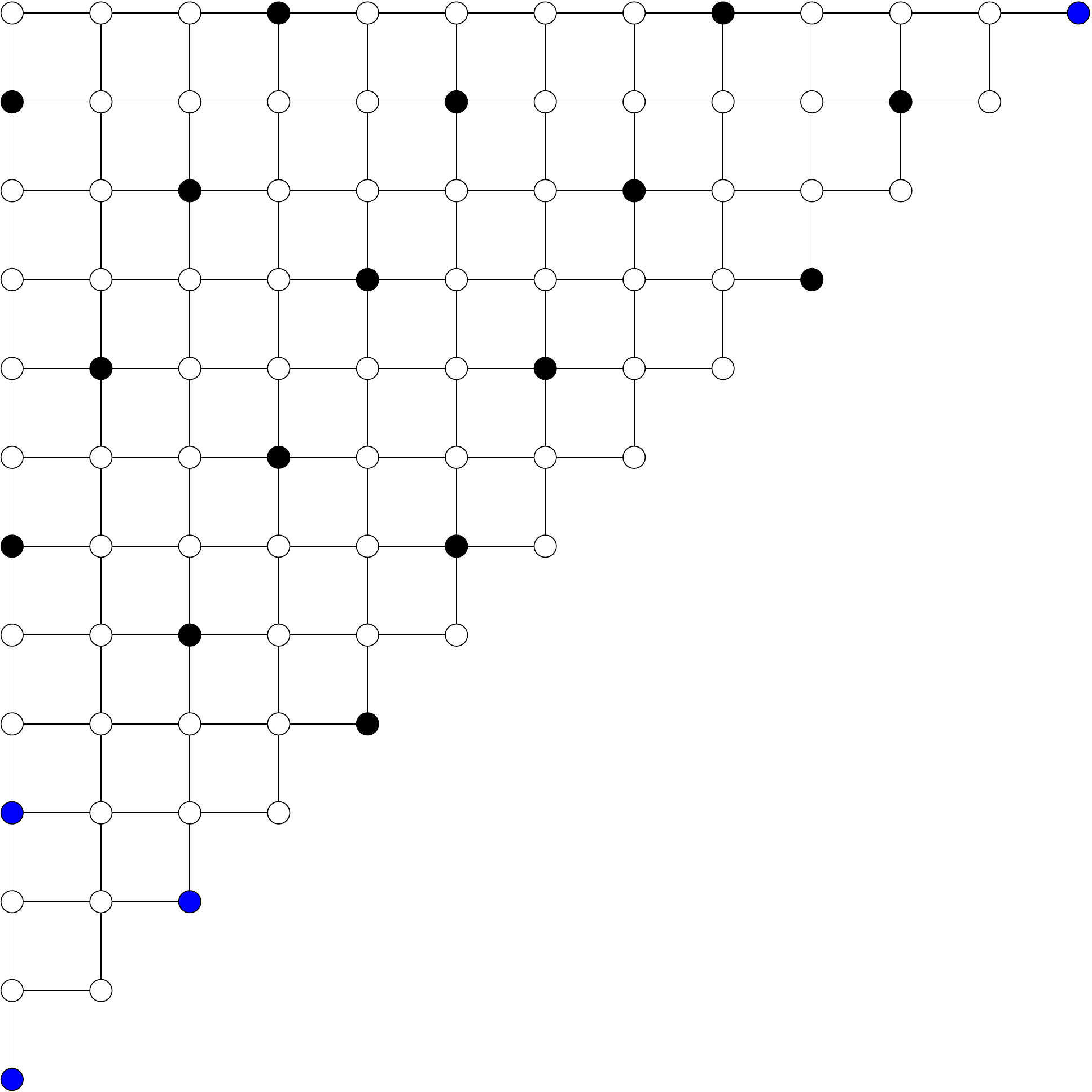}\\
(a)&(b)&(c)\\
\includegraphics[width=.250\textwidth]{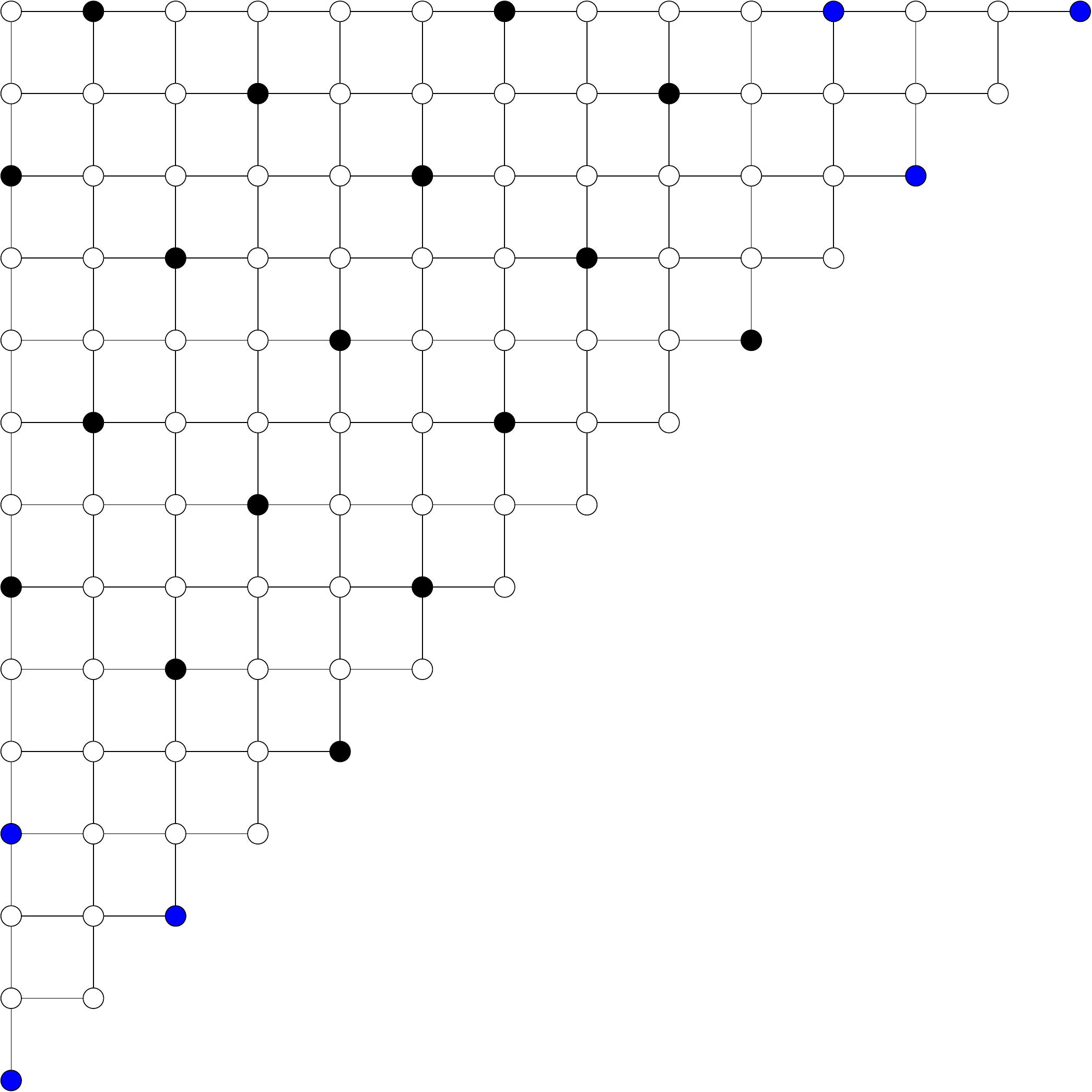}&\includegraphics[width=.250\textwidth]{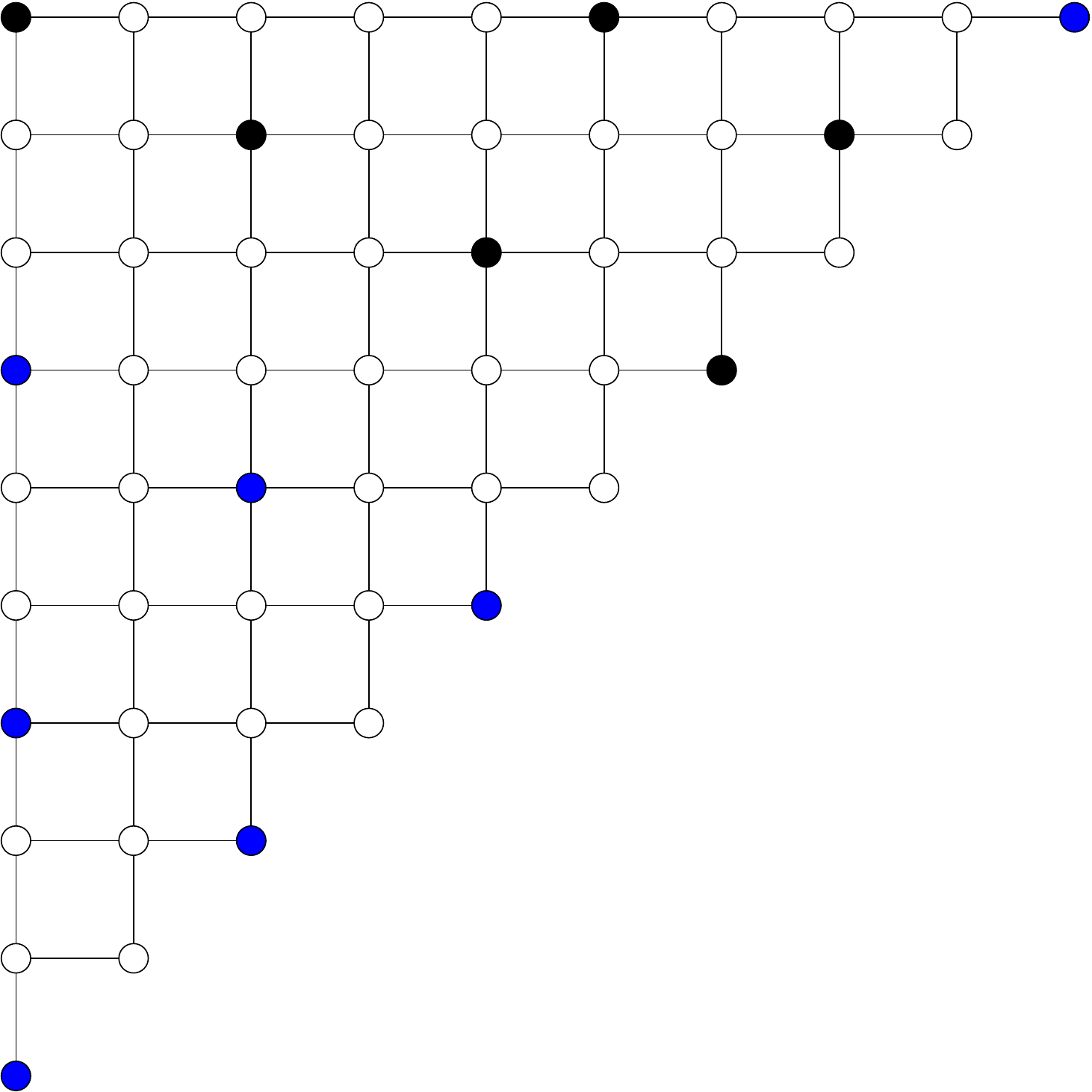}\\
(d)&(e)\\
\end{tabular}
\caption{\small Packing sets $A_{n,t}$ for $n=11, 12, 13, 14$ and $10$ are showed in figures (a), (b), (c), (d) and (e), respectively.}
\label{fig:construction}
\end{center}
\end{figure}

Using Proposition~\ref{prop:H} (ii), it is easy to check that $A_{n, t}$ is a packing set for $T(n)$ whenever $t = n \bmod{5}$. Let $j \in \{5,\dots,n\}$. From the definitions of $A_{n,t}, T^{j-4,j}(n)$ and  Proposition \ref{values-franja} it follows that:

\[
|A_{n,1}|=2+\sum_{i=1}^{(n-1)/5}|T^{(5i+1)-4,5i+1}(n)\cap f^{-1}(4)|=2+\sum_{i=1}^{(n-1)/5}(5i-1)=\frac{1}{10}(n^2+n+18).
\]

\[
|A_{n,2}|=2 + \sum_{i=1}^{(n-2)/5}|T^{(5i+2)-4, 5i+2}(n)\cap f^{-1}(0)|=2+\sum_{i=1}^{(n-2)/5}5i=\frac{1}{10}(n^2+n+14).
\]

\[
|A_{n,3}|=3 + \sum_{i=1}^{(n-3)/5}|T^{(5i+3)-4, 5i+3}(n)\cap f^{-1}(0)|=3+\sum_{i=1}^{(n-3)/5}(5i+1)=\frac{1}{10}(n^2+n+18).
\]

\[
|A_{n,4}|=4+\sum_{i=1}^{(n-4)/5}|T^{(5i+4)-4, 5i+4}(n)\cap f^{-1}(0)|=4+\sum_{i=1}^{(n-4)/5}(5i+2)=\frac{1}{10}(n^2+n+20).
\]

\[
|A_{n,0}|=5+\sum_{i=2}^{n/5}|T^{5i-4,5i}(n)\cap f^{-1}(1)|=5+\sum_{i=2}^{n/5}(5i-2)=\frac{1}{10}(n^2+n+20).
\]
 $\square$
  
\section{The sequence $\rho(T(n))$ is upper bounded by $a(n)$}
We recall that for  $n \in \{1, \dots, 49\}$, the exact value of $\rho(T(n))$ is given by the sequence  A085680$(n+1)$ in OEIS \cite{oeis}, and that such values of $A085680(n+1)$ coincide with $a(n)$ for $n\geq 6$. Our aim in this section is to show Lemma \ref{lemma:upper} by proving that $\rho(T(n))$ is upper bounded by $a(n)$, and so confirming that the generating function of $A085680(n+1)$ proposed by Rob Pratt \cite{oeis} is right. 

\subsection{Auxiliar results}
 In this subsection we recall some known results about the packing number of graphs and we prove several lemmas. 
  All these results will be used in the proof of Lemma \ref{lemma:upper}. 
 
As usual, for $p$ and $q$ positive integers we will use $G_{p,q}$ to denote the {\em grid graph} of size $p\times q$, with $p \leq q$. The exact values of $\rho(G_{p,q})$ were determined by  D. C. Fisher~\cite{fisher} and are as follows:  

\begin{equation}\label{eq:Gpq}
\rho(G_{p,q})= \left\{ \begin{array}{lcc}
             \lceil \frac{(p+1)q}{6}\rceil &   if & p\in\lbrace 1,2,3\rbrace, \\
             \\ \lceil \frac{6q}{7}\rceil &  if & p=4\text{ and }q\not\equiv  1\pmod 7, \\
             \\ \lceil \frac{6q}{7}\rceil+1 &  if & p=4\text{ and }q\equiv  1\pmod 7, \\
	   \\ 10 & if & (p,q)=(7,7),\\
             \\ \lceil \frac{pq+2}{5}\rceil &  if & p\in\lbrace 5,6,7\rbrace\text{ and }(p,q)\neq (7,7), \\
	  \\ 17 & if & (p,q)=(8,10),\\
	  \\ \lceil \frac{pq}{5}\rceil &  if & p\geq 8 \text{ and } (p,q)\neq (8,10).
            \end{array}
   \right.
\end{equation}

As we shall see, this result on $\rho(G_{p,q})$ will be used in most of the proofs given in the rest of the paper.  

Our next observation follows directly from the definition of the packing number of a graph, and is one of our main tools in this section. Indeed, a strategy used in many proofs of this section is as follows. We partition the vertex set of the graph $G$ under consideration into two or more subsets, and then we obtain upper bounds for the packing number of $G$  by giving upper bounds for the packing numbers of the subgraphs induced by such subsets of $V(G)$. We usually use this strategy without explicit mention.  

\begin{obs}\label{obs1}
Let $\lbrace V_1,V_2\rbrace$ be a partition of $V(G)$. If $G_1$ and $G_2$ are the subgraphs of $G$ induced by $V_1$ and $V_2$, respectively, then 
\[
\rho(G)\leq \rho(G_1)+\rho(G_2).
\]
\end{obs}

The following corollary is a consequence of  Observation \ref{obs1} and the facts that $\{V(T(n-5)), V(T^5(n))\}$ and $\{V(T(n-10)), V(T^{10}(n))\}$ are partitions of $V(T(n))$ for $n>10$. 

\begin{cor}\label{cota-superior-arriba}
Let $n>10$ be an integer. Then $\rho(T(n))\leq \rho(T(n-5))+\rho(T^5(n))$ and $\rho(T(n))\leq \rho(T(n-10))+\rho(T^{10}(n))$.
\end{cor}

\begin{lem}\label{maximofranja10}
Let $m$ and $n$ be integers. Then,
\begin{itemize}
\item[(i)] $\rho (T^5(m))=m-1$, for $m\geq 5$;
\item[(ii)] $\rho(T^{10}(n))=2n-8$, for $n\geq 12$. 
 \end{itemize}
\end{lem}

\begin{proof}
Let $m$ and $n$ be as in the statement of the lemma. 
\vskip 0.4cm
\noindent({\bf A}) First we show that $\rho (T^5(m))\geq m-1$ and that $\rho (T^{10}(n))\geq 2n-8$.  Let $r\in \{m,n\}$.
If $r=m$ we let $\ell=5$, and if $r=n$ then we let $\ell=10$. It is enough to exhibit a packing set for $T^{\ell}(r)$ with the required size. Consider  the homomorphism $f \colon \Z^2\to \Z_5$ defined in Section \ref{notacion}, namely $f(x,y)=(x+2y) \bmod{ 5 }$. Let $c=(3r-4)\bmod {5}$. Let $S_r:=T^{\ell}(r)\cap f^{-1}(c)$. From Proposition \ref{prop:H} (iii), we know that $S_r$ is a packing set of $T^{\ell}(r)$. Moreover, since $(r,r)\in T^{\ell}(r)\setminus f^{-1}(c)$, and no element in $\{(r-1,r), (r-2,r),(r-1,r-1)\}$ belongs to $S_r$, then $S_r\cup \{(r,r)\}$ is a packing set for $T^{\ell}(r)$.  
 As $S_m=T^5(m)\cap f^{-1}(c)$  and $S_n=(T^5(n)\cap f^{-1}(c))\cup (T^{n-9,n-5}_{1,n-5}(n)\cap f^{-1}(c))$, by Proposition \ref{values-franja} we have that 
 $|S_m|=m-2$ and $|S_n|=(n-2)+((n-5)-2)$. Then $S_r\cup \{(r,r)\}$ is the required packing set.

\vskip 0.4cm

\noindent({\bf B}) Now we show that $\rho (T^5(m))=m-1$.  In view of (A), it is enough to show that $\rho (T^5(m))\leq m-1$.  We proceed by induction on $m$.
The cases $m = 5,6,7$ are easy to verify. Now we make the inductive assumption that the statement holds for any $k \in \{5,6,\ldots,m -1\}$, and show that it holds for 
$m\geq 8$.  As $T^5_{i,m}(m)\simeq T^5(m-i+1)$ for $i=2,3, \ldots , m-4$, by induction we have that $\rho(T^5_{i,m}(m))=m-i$. Let $S$ be a maximum packing set of $T^5(m)$ and let $S':=S^5_{2,m}(m)$. Because $S'$ is a packing set for $T^5_{2,m}(m)$ and $\rho(T^5_{2,m}(m))=m-2$, then $|S'|\leq m-2$. Since the sets of vertices $V(T^5_1(m))$ and $V(T^5_{2,m}(m))$ form a partition of the vertex set of $T^5(m)$, and $T^5_1(m)\simeq P_5$, then $|S|\leq |S'|+\rho(P_5)= |S'|+2$. We may assume that $|S'|=m-2$, as otherwise we are done. For $j\in\{1,2,3\}$, let $s_j$ be the number of vertices in $S^5_j(m)$. Clearly, $s_j\in \{0,1,2\}$. 

\begin{itemize}
\item [1.] If $s_2=0$,  then the assumption $|S'|=m-2$ implies that $|S^5_{3,m}(m)|=m-2$, contradicting that $\rho(T^5_{3,m}(m))=m-3$.
\item [2.] If $s_2=1$ we have two cases.
\begin{itemize}
\item [2.1] If $s_3\geq 1$, it is an easy exercise to check that $s_1$ must be at most $1$, and hence  $|S|=s_1+|S'|\leq m-1$.
\item [2.2] If $s_3=0,$ then $|S^5_{4,m}(m)|=|S'|-s_2-s_3=m-3$, contradicting that $\rho(T^5_{4,m}(m))=m-4$.
\end{itemize}
\item [3.] If $s_2=2$, then $s_1\leq 1$, and hence $|S|=s_1+|S'|\leq m-1.$
\end{itemize}

\noindent({\bf C}) Now we show that $\rho (T^{10}(n))=2n-8$. Again, in view of (A), it is enough to show that $\rho (T^{10}(n))\leq 2n-8$. 
We have verified by computer that  $\rho(T^{10}(n))\leq 2n-8$ holds for $n\in \{12,\ldots ,18\}$. Thus we will assume that $n\geq 19$. 

Let $S$ be a maximum packing set of $T^{10}(n)$. Note that if $|S^{10}_{1,n-10}(n)|\leq 2(n-10)-1$ or $|S^{10}_{n-9,n}(n)|\leq 13-1$, then Observation \ref{obs1} implies the required inequality: 
\[
|S|=|S^{10}_{1,n-10}(n)|+|S^{10}_{n-9,n}(n)|\leq 2(n-10) +13 - 1= 2n - 8.
\]

Thus we can assume that $|S^{10}_{1,n-10}(n)|\geq 2(n-10)$ and $|S^{10}_{n-9,n}(n)|\geq 13$. Moreover, since $T^{10}_{n-9,n}(n)\simeq T(10)$, then $|S^{10}_{n-9,n}(n)|\leq \rho(T(10))=13$ by A085680 in \cite{oeis}. Similarly, since $T^{10}_{1,n-10}(n)\simeq G_{10,n-10}$, then $|S^{10}_{1,n-10}(n)|\leq 2(n-10)$ by  Equation (\ref{eq:Gpq}). Therefore we can assume that $|S^{10}_{1,n-10}(n)|=2(n-10)$ and $|S^{10}_{n-9,n}(n)|=13$. 

By computer we verified that (up to isomorphism) the four packing sets of $T^{10}_{n-9,n}(n)$ given in Figure \ref{fig:unicosT10} are the only sets attaining the maximum size, namely $13$. 
 
 \begin{figure}[H]
\begin{center}
\includegraphics[width=.9\textwidth]{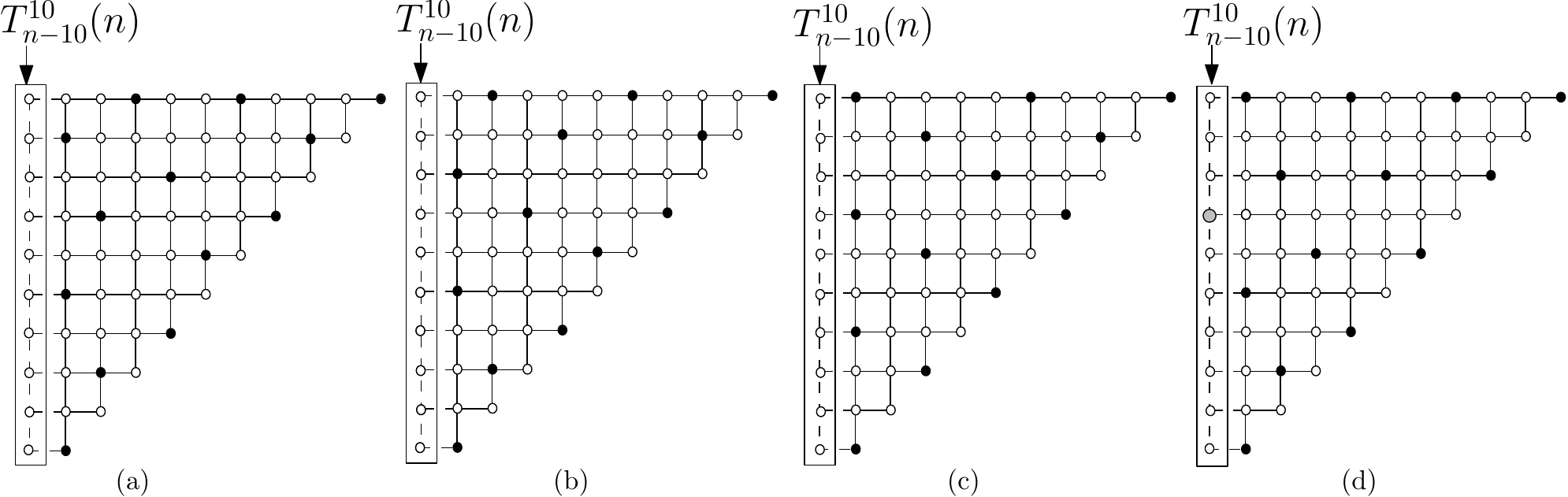}
\caption{\small These are (up to isomorphism) all the packing sets of $T_{10}$ with $13$ vertices.}
\label{fig:unicosT10}
\end{center}
\end{figure}
 
Observe that the packing sets in Figures \ref{fig:unicosT10} (a), (b) and (c) do not allow vertices of $S$ in $T^{10}_{n-10}(n)$ (i.e., in the vertical dotted lines).
Similarly, note that the packing set in Figure  \ref{fig:unicosT10} (d) allows at most 1 vertex of $S$ in $T^{10}_{n-10}(n)$, namely the gray vertex whose coordinates are $(n-10, n-3)$. Therefore we have that  $|S^{10}_{n-10}(n)|\leq 1$.

First, we suppose that $n=19$. From our assumption, we know that $|S_{1,9}^{10}(19)|=18$. If $|S^{10}_{9}(19)|=0$, then $|S^{10}_{1, 8}(19)|=18$ which contradicts $\rho(T^{10}_{1, 8}(19))=17$ (Equation (\ref{eq:Gpq})). Now, suppose that $|S^{10}_{9}(19)|=1$. Then $|S^{10}_{1, 8}(19)|=17$ and $S$ is as in Figure \ref{fig:unicosT10} (d).
 The last assertion implies that $|S^{10}_{8}(19)|\leq 2$.  On the other hand, we have verified by computer that $|S^{10}_{1, 8}(19)|=17$ implies $|S^{10}_{8}(19)|\geq 3$, a contradiction.    

Now we suppose that $n\geq 20$. Then $|S_{1,n-10}^{10}(n)|=2(n-10)$. Since $|S^{10}_{n-10}(n)|\leq 1$, then $|S^{10}_{1,n-11}(n)|\geq 2(n-10)-1,$ which contradicts that $\rho(T^{10}_{1,n-11}(n))=2(n-10)-2$ whenever $n-11\geq 9$ (Equation (\ref{eq:Gpq})). 

\end{proof}

The value $\rho(T^{10}(11))=15$ was determined by computer and it does not satisfy the regularity stated in Lemma \ref{maximofranja10} (ii). 

\begin{cor}\label{U5menormax}
Let $n\geq 5$ be an integer and let $S$ be a packing set of $T^5(n)$. If $|S^5_{1,i}(n)|\leq i-1$ for some $i\in \{1, \dots, n-4\}$, then $|S|<n-1$.
\end{cor}
\begin{proof}
Note that $|S|=|S^5_{1,i}(n)|+|S^5_{i+1,n}(n)|$ and that $T^5_{i+1,n}(n)\simeq T^5 (n-i)$. Then from Lemma~\ref{maximofranja10} (i) we have that $|S^5_{i+1,n}(n)|\leq (n-i)-1$, and so  $|S|\leq i-1+(n-i)-1=n-2 < n-1$, as desired.
\end{proof}

\begin{cor}\label{lemaparticion}
Let $n\geq 27$ and let $S$ be a packing set of $T^{10}(n)$. Let $i\in \{0,1, 2\}$ and $j+i \in\{12, \dots, n-12\}$. If $|S^{10}_{j,j+i}(n)|\leq 2i+1$, then $|S|< 2n-8$.
\end{cor}
\begin{proof}
Clearly,  $|S|=|S^{10}_{1,j-1}(n)|+|S^{10}_{j,j+i}(n)|+|S^{10}_{j+i+1, n}(n)|$. From Equation (\ref{eq:Gpq}) we know that $|S^{10}_{1,j-1}(n)|\leq 2(j-1)$, and by Lemma \ref{maximofranja10} (ii) that
 $|S^{10}_{j+i+1,n}(n)|\leq 2(n-(j+i))-8.$ Therefore, $|S|\leq 2(j-1)+(2i+1)+2(n-(j+i))-8=2n-9 <2n-8$, as desired. 
\end{proof} 

The proof of our next proposition is a routine exercise. 
\begin{prop}\label{obsno32}
Let $n\geq 27$, $j\in \{1,\ldots , n-12\}$ and $k\in \{1,2\}$. If $S$ is a packing set of $T^{5k}(n)$ and  two of $|S^{5k}_j(n)|, |S^{5k}_{j+1}(n)|,|S^{5k}_{j+2}(n)|$ are equal to $k$, then the other one is at most $k$. 
\end{prop}

\begin{lem}\label{lema12013}
Let $n\geq 27$ be an integer. Let $S$ be a maximum packing set of $T^{10}(n)$. Then $12\leq |S^{10}_{n-9,n}(n)|\leq 13$. 
\end{lem}
\begin{proof} From Lemma \ref{maximofranja10} (ii) we know that $|S|=2n-8$.  Since $T^{10}_{n-9,n}(n)\simeq T(10)$ and  $\rho(T(10))=13$, by A085680 in \cite{oeis}, then $|S^{10}_{n-9,n}(n)|\leq 13$. Suppose now that $|S^{10}_{n-9,n}(n)|\leq 11$. As $|S|=|S^{10}_{1,n-10}(n)|+|S^{10}_{n-9,n}(n)|$, by Equation~(\ref{eq:Gpq}) we have $|S| \leq 2(n-10)+11<2n-8$, a contradiction. 
\end{proof}

\begin{prop}\label{proposicionunounofranjas}
Let $n \geq 27$ and $j\in \{10,\ldots ,n-12\}$. If $S$ is a maximum packing set of $T^{10}(n)$, 
then $|S^{10}_j(n)|=2$,  $|S^5_j(n)|=1$ and $|S^{n-9, n-5}_j(n)|=1$.
\end{prop}

\begin{proof} Let $j,n$ and $S$ be as in the statement of the proposition. Then $|S|=2n-8$ by Lemma \ref{maximofranja10} (ii). Also note that $|S^{5}_j(n)|\in \{0,1,2\}$ and that $|S^{10}_j(n)|\in \{0,1,2,3,4\}$.
We proceed by showing each of the  conclusions of the proposition separately. 
\vskip 0.2cm
{\bf Claim A} $|S^{10}_j(n)|=2$.
\vskip 0.2cm
{\em Proof of Claim A.} We divide the  proof of this claim depending on the values of $j$.  
\vskip 0.2cm
{\bf Case 1} $j\in \{12,13,\ldots ,n-14\}$. In each of the next subcases, we derive a contradiction from the assumption that $|S^{10}_j(n)|\not=2$.
\begin{enumerate}

\item [{\bf Subcase 1.1}] $|S^{10}_j(n)|\leq 1$. Then Corollary \ref{lemaparticion} (with $i=0$) implies 
$|S|<2n-8=\rho(T^{10}(n))$, a contradiction.
\item [{\bf Subcase 1.2}] $|S^{10}_j(n)|=3$. We analyze three subcases separately. 

\begin{itemize}
\item [{\bf Subcase 1.2.1}] Neither $(j,n)$ nor $(j,n-9)$ belong to $S$. In this case, it is easy to see that $|S^{10}_{j-1}(n)\cup S^{10}_{j+1}(n)|\leq 1$. Therefore, $|S^{10}_{j-1,j+1}(n)|\leq 4$. By Corollary \ref{lemaparticion}, with $i=2$, we have $|S|<2n-8,$ a contradiction.

\item [{\bf Subcase 1.2.2}]  Exactly one of  $(j,n)$ or $(j,n-9)$ belongs to $S$. Again, it is not hard to see that the current assumptions on $S$ imply $|S^{10}_{j-1}(n)\cup S^{10}_{j+1}(n)|\leq 2$, and hence  $|S^{10}_{j-1,j+1}(n)|\leq 5$. By Corollary \ref{lemaparticion}, with $i=2$, we have $|S|<2n-8$, a contradiction.

\item [{\bf Subcase 1.2.3}] Both vertices  $(j,n)$ and $(j,n-9)$ are in $S$. Similarly, the current assumptions on $S$ imply $|S^{10}_{j-1}(n)\cup S^{10}_{j+1}(n)|\leq 3$. If $|S^{10}_{j-1}(n)\cup S^{10}_{j+1}(n)|\leq 2$, then $|S^{10}_{j-1,j+1}(n)|\leq 5$ and we proceed as in Subcase 1.2.2. Then we suppose that 
$|S^{10}_{j-1}(n)\cup S^{10}_{j+1}(n)|=3$. From $(j,n), (j,n-9)\in S^{10}_j(n)$ it follows that $S^{10}_{j-1}(n)\subset T^{n-7,n-2}_{j-1}(n)$ and  that $S^{10}_{j+1}(n)\subset T^{n-7,n-2}_{j+1}(n)$. Since $T^{n-7,n-2}_{j-1}(n)\simeq T^{n-7,n-2}_{j+1}(n) \simeq P_6$ and $\rho(P_6)=2$, then we have that 
 exactly one of $|S^{10}_{j-1}(n)|$ or $|S^{10}_{j+1}(n)|$ is equal to $2$. We only analyze the case $|S^{10}_{j-1}(n)|=2$, because the other case is totally analogous. If $|S^{10}_{j-1}(n)|=2$, then  $|S^{10}_{j+1}(n)|=1$ and we only have four possible configurations for $S^{10}_{j-1,j+1}(n),$ see Figure \ref{unicas2:tiles} ($a$). We note that any of these four configurations implies $|S^{10}_{j-2}(n)|\leq 1$. The last inequality and Corollary \ref{lemaparticion} (with $i=0$) imply $|S|<2n-8$, a contradiction.

\end{itemize}

\item [{\bf Subcase 1.3}] $|S^{10}_j(n)|=4$. Then $S^{10}_j(n)$ must be as shown in Figure~\ref{unicas2:tiles} ($b$). This implies $|S^{10}_{j-1}(n)\cup S^{10}_{j+1}(n))|=0$, and so 
$|S^{10}_{j-1,j+1}(n)|=4$.  By Corollary \ref{lemaparticion}, with $i=2$, we have $|S|<2n-8$, a contradiction.
\end{enumerate}

\begin{figure}[H]
\begin{center}
\includegraphics[width=.55\textwidth]{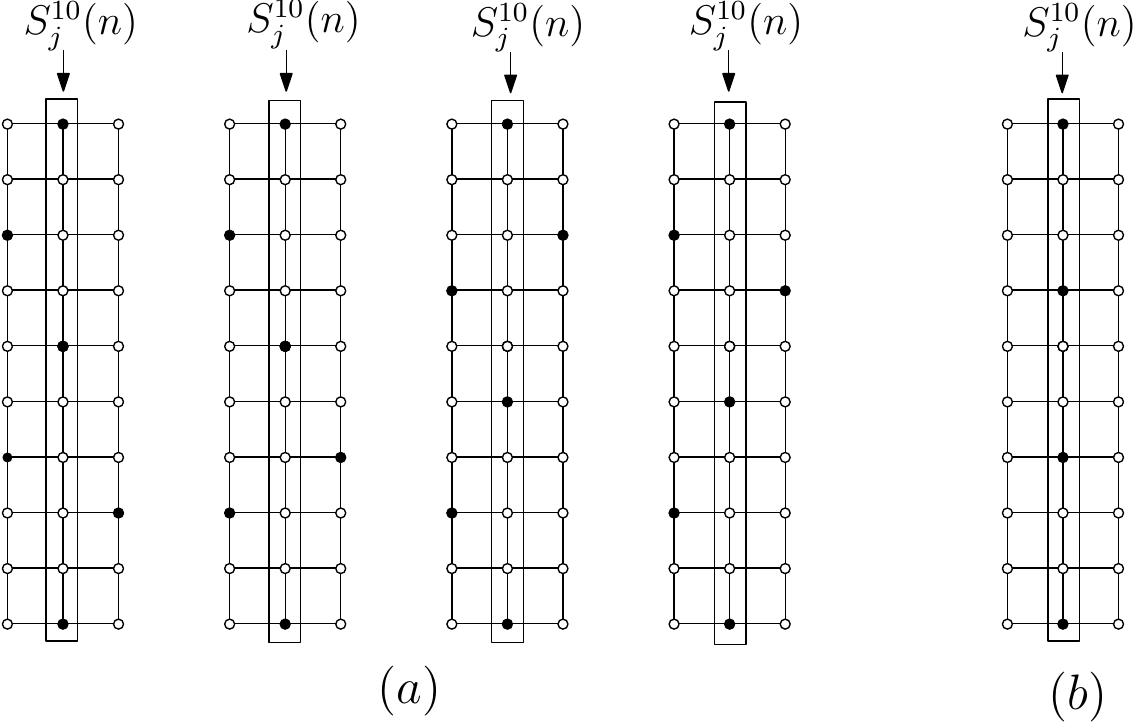}
\caption{\small  (a) These four are the only configurations satisfying $|S^{10}_{j-1}(n)|=2, |S^{10}_j(n)|=3$ and $|S^{10}_{j+1}(n)|=1$. (b) This is the only one configuration 
	satisfying $|S^{10}_j(n)|=4$.}
\label{unicas2:tiles}
\end{center}
\end{figure}

{\bf Case 2}  $j\in \{10,11, n-13,n-12\}$. We need to show that $|S^{10}_j(n)|=2$.

\begin{enumerate}
\item [{\bf Subcase 2.1}]  $j=11$. By Case 1 we know that $|S^{10}_{12}(n)|=|S^{10}_{13}(n)|=2$. These equalities and Proposition~\ref{obsno32} imply 
 $|S^{10}_{11}(n)|\leq 2$. On the other hand, from Lemma~\ref{maximofranja10} (ii) and Equation (\ref{eq:Gpq}) we known that $|S^{10}_{12,n}(n)|\leq 2(n-11)-8$ and 
 $|S^{10}_{1,10}(n)|\leq 20$, respectively. If $|S^{10}_{11}(n)|< 2$, then we have 
\[
|S|= |S^{10}_{1,10}(n)|+|S^{10}_{11}(n)|+|S^{10}_{12,n}(n)|< 20+2+2(n-11)-8=2n-8,
\] a contradiction. Then $|S^{10}_{11}(n)|= 2$.

\item [{\bf Subcase 2.2}] $j=10$. By Case 1 and Subcase 2.1 we know that $|S^{10}_{11}(n)|=|S^{10}_{12}(n)|=2$. These equalities and 
Proposition~\ref{obsno32} imply $|S^{10}_{10}(n)|\leq 2$. If $|S^{10}_{10}(n)|< 2$, then from Lemma~\ref{maximofranja10} (ii) and Equation (\ref{eq:Gpq}) it follows that
\[
|S|= |S^{10}_{1,9}(n)|+|S^{10}_{10}(n)|+|S^{10}_{11,n}(n)| < 18+2+2(n-10)-8=2n-8,
\]  a contradiction.  Then $|S^{10}_{10}(n)|= 2$.

\item [{\bf Subcase 2.3}] $j=n-13$. By Case 1 we know that $|S^{10}_{n-15}(n)|=|S^{10}_{n-14}(n)|=2$. These equalities and Proposition~\ref{obsno32} imply 
 $|S^{10}_{n-13}(n)|\leq 2$. If $|S^{10}_{n-13}(n)|< 2$, then from Lemma~\ref{maximofranja10} (ii) and Equation (\ref{eq:Gpq}) it follows that 
\[
|S|= |S^{10}_{1,n-14}(n)|+|S^{10}_{n-13}(n)|+|S^{10}_{n-12,n}(n)|< 2(n-14)+2+2(13)-8=2n-8,
\] a contradiction. Then $|S^{10}_{n-13}(n)|=2$.

\item [{\bf Subcase 2.4}] $j=n-12$. By Case 1 and Subcase 2.3 we know that $|S^{10}_{n-14}(n)|=|S^{10}_{n-13}(n)|=2$.  These equalities and Proposition~\ref{obsno32} 
imply $|S^{10}_{n-12}(n)|\leq 2$.  If $|S^{10}_{n-12}(n)|< 2$, then from Lemma~\ref{maximofranja10} (ii) and Equation (\ref{eq:Gpq}) it follows that
\[
|S|= |S^{10}_{1,n-13}(n)|+|S^{10}_{n-12}(n)|+|S^{10}_{n-11, n}(n)| < 2(n-13)+2+2(12)-8=2n-8,
\] a contradiction. Then $|S^{10}_{n-12}(n)|=2$. $\triangle$
\end{enumerate} 

\vskip 0.2cm
{\bf Claim B} $|S^5_j(n)|=1$ and $|S^{n-9, n-5}_j(n)|=1$.
\vskip 0.2cm
{\em Proof of Claim B.}  We analyze three cases separately depending on the value of $j$.

\begin{itemize}
\item [{\bf Case 1}]  $j\in \{12,\ldots ,n-14\}$. Let $G_{10,5}$  be the $10\times 5$ grid graph with vertex set $\{\{1,\ldots ,10\}\times \{1,\ldots ,5\}\}$. 
The statement for this case was verified as follows: we find by computer all the packing sets of $G_{10,5}$ which have exactly two packing vertices in each of its $5$ columns, and then we verify that in all these sets (namely 54), the $y$-coordinate of exactly one of the two packing vertices in the middle column of $G_{10,5}$ is in $\{1,\ldots ,5\}$ and the $y$-coordinate of the other one is in $\{6,\ldots ,10\}$. This implies the desired result for $j \in \{12, \ldots, n-14\}$ because from Claim A we know that $S$ has exactly two packing vertices in each column of $T^{10}_{10,n-12}(n)$.

\item [{\bf Case 2}]  $j\in \{11, n-13\}$. If $j=11$ let $r:=1$, and if $j=n-13$ let $r:=-1$.  From Case 1 we know that $|S^5_{j+r}(n)|=|S^5_{j+2r}(n)|=1$ and 
$|S^{n-9, n-5}_{j+r}(n)|=|S^{n-9, n-5}_{j+2r}(n)|=1$. These and Proposition \ref{obsno32} imply  $|S^5_j(n)|\leq 1$ and $|S^{n-9, n-5}_j(n)|\leq 1$, respectively. 
The last inequalities and Claim A imply $|S^5_j(n)|=|S^{n-9, n-5}_j(n)|=1$, as desired. 

\item [{\bf Case 3}]  $j\in \{10,n-12\}$.  Similarly, if $j=10$ let $r:=1$, and if $j=n-12$ let $r:=-1$. From  Case 1 and Case 2 we know that $|S^5_{j+r}(n)|=|S^5_{j+2r}(n)|=1$ and
 $|S^{n-9, n-5}_{j+r}(n)|=|S^{n-9, n-5}_{j+2r}(n)|=1$. These and Proposition \ref{obsno32} imply $|S^5_j(n)|\leq 1$ and $|S^{n-9, n-5}_{j}(n)|\leq 1$, respectively. The last inequalities and Claim A imply $|S^5_j(n)|=1$ and $|S^{n-9, n-5}_{j}(n)|=1$, as desired. $\triangle$
\end{itemize} 
\end{proof}

\begin{lem}\label{claim1}
Let $n\geq 27$ be an integer and let $S$ be a packing set of $T^{10}(n)$ such that $|S|=2n-8$ and $|S^5(n)|=n-1$. Then 

\begin{itemize}
\item[(i)] $|S^5_{1,9}(n)|=9$. 
\item[(ii)] $|S^5_i(n)|=1$, for $i\in\{10, \dots, n-4\}$. 
\item[(iii)] $|S^5_{1, i}(n)|=i$, for  $i\in \{9,\ldots ,n-4\}$.
\end{itemize}
\end{lem}

\begin{proof} We show each of these three conclusions separately.
\vskip 0.3cm
First we show (i). From Equation (\ref{eq:Gpq}) we know that $\rho(T^5_{1,i}(n))=i+1$, for any $i\in \{1,2,\ldots ,n-4\}$. Then $|S^5_{1,9}(n) | \leq 10$. If $|S^5_{1,9}(n)| \leq 8$, then Corollary \ref{U5menormax} implies $|S^5(n)| < n-1$, a contradiction. Hence $9 \leq |S^5_{1,9}(n)|\leq 10$. We will derive a contradiction by assuming that $|S^5_{1,9}(n)|=10$. 

From Proposition \ref{proposicionunounofranjas} we know that $|S^5_{10}(n)|=|S^5_{11}(n)|=1$. These equalities and Proposition~\ref{obsno32} imply
  $|S^5_9(n)|\leq 1$. On the other hand, since $10=|S^5_{1,9}(n)|=|S^5_{1,8}(n)|+|S^5_9(n)|$ and $|S^5_{1,8}(n)|\leq 9$ by Equation (\ref{eq:Gpq}), then
  we must have  $|S^5_{1,8}(n)|=9$ and $|S^5_9(n)|=1$. Note that by a repeated application of the reasoning in this paragraph we can conclude that $|S^5_{j}(n)|=1$ for $j=8,7,\dots, 1$, and so $\sum _{j=1}^{9}|S^5_j(n)|=9$. Since $|S^5_{1,9}(n)|=\sum _{j=1}^{9}|S^5_j(n)|$, then $10=9$, a  contradiction.  
  
 \vskip 0.3cm 
 
Now we show (ii).  In view of Proposition \ref{proposicionunounofranjas}, we only need to show that $|S^5_i(n)|=1$ for every $i \in \{n-11,n-10, \dots, n-4\}$.    
From Proposition \ref{proposicionunounofranjas} we know that $|S^5_{10,n-12}(n)|=\sum_{j=10}^{n-12}|S^5_j(n)|= n-21$. This equality and Lemma \ref{claim1} (i) imply that
\[
|S^5_{1,n-12}(n)|=|S^5_{1,9}(n)|+|S^5_{10, n-12}(n)| = 9 + (n - 21) = n-12.
\] 
Since by Proposition \ref{proposicionunounofranjas} $|S^5_{n-13}(n)|=|S^5_{n-12}(n)|=1$, then  $|S^5_{n-11}(n)|\leq 1$ by Proposition~\ref{obsno32}. 
If $|S^5_{n-11}(n)|=0$, then Lemma~\ref{claim1} (i) and Lemma~\ref{maximofranja10} (i) imply
$|S^5(n)|=|S^5_{1,n-12}(n)|+|S^5_{n-11}(n)|+|S^5_{n-10, n}(n)| \leq (n-12)+0+10=n-2,$ 
which contradicts the hypothesis that $|S^5(n)|=n-1$. Therefore $|S^5_{n-11}(n)|=1$. So we have $|S^5_{n-12}(n)|=|S^5_{n-11}(n)|=1$, and by Proposition~\ref{obsno32} we have that $|S^5_{n-10}(n)|\leq 1.$ As before, by a repeated application of the reasoning in this paragraph we can deduce that  
$|S^5_j(n)|=1$ for $j=n-10,n-9, \dots, n-4$, as required.

Finally, we note that the assertion in (iii) follows directly from (i) and (ii). 
\end{proof}


\subsection{Proof of Lemma \ref{lemma:upper}}

We proceed by induction on $n$. Since $a(n)=\rho(T(n))=A085680(n+1)$ for every $n\in \{1,\ldots ,49\}$, we can assume that $n \geq 50$ and that the assertion holds for every $k < n$.
 By the recursion given in Observation~\ref{formula:rec}, it is enough to show that $\rho(T(n)) \leq a(n-5)+n-2$.
 
 On the other hand, we recall the following isomorphism relations between elements of $\{T(i)\}_{i\in \mathbb Z^+}$ and the given subgraphs of $T(n)$: $T(n-5)\simeq T^{1,n-5}_{1,n-5}(n)$, $T(n-10) \simeq T^{1,n-10}_{1,n-10}(n)$, $T(10)\simeq T^{10}_{n-9,n}(n)$ and $T(5)\simeq T^{5}_{n-4,n}(n)$. Sometimes throughout this proof, for brevity of notation, for $i<n$ we will use the $T(i)$ to denote its corresponding isomorphic $T's$ subgraphs.
 
Let $S$ be a maximum packing set of $T(n)$. Thus $\rho(T(n))=|S|$. We recall that $S^5(n)=T^5(n)\cap S$ and that $S^{10}(n)=T^{10}(n)\cap S$. 
If $|T(n-10) \cap S|\leq a(n-10)-1$ or $|S^{10}(n)| \leq (2n-8)-1$, then Observation~\ref{formula:rec} implies the required inequality:
\[
|S| = |T(n-10)\cap S| + |S^{10}(n)| \leq a(n-10)+(2n-8)-1 = a(n-5)+n-2.
\]
Similarly, if $|T(n-5) \cap S |\leq a(n-5)-1$ or $|S^5(n)|\leq (n-1)-1$ we obtain 
\[
|S|\leq a(n-5) + (n-1) -1 = a(n-5)+n-2.
\] 
Thus we can assume that $|T(n-5) \cap S |\geq a(n-5)$, $|T(n-10) \cap S |\geq a(n-10)$, $|S^5(n)|\geq n-1$ and $|S^{10}(n)|\geq 2n-8$. 
Moreover, since $S^5(n)$ and $S^{10}(n)$ are packing sets of 
$T^5(n)$ and $T^{10}(n)$ respectively, then from Lemma \ref{maximofranja10} we must have that

\begin{equation}\label{franjas}
|S^5(n)|=n-1 \text{ and }  
|S^{10}(n)|=2n-8.
\end{equation} 

In other words, $S^5(n)$ and 
$S^{10}(n)$ must be maximum packing sets of $T^5(n)$ and $T^{10}(n)$, respectively. Similarly, since $T(n-5)\cap S$ and
$T(n-10) \cap S$ are packing sets of  $T(n-5)$ and $T(n-10)$ respectively, then from the induction hypothesis we have that 

\begin{equation}\label{induction}
|T(n-5) \cap S | = a(n-5) \text{ and }  
|T(n-10) \cap S | = a(n-10).
\end{equation}

From now on, we derive several contradictions from Equations (\ref{franjas}) and  (\ref{induction}). 

From $|S^{10}(n)|=2n-8$ and Lemma~\ref{lema12013} it follows that $12\leq |S^{10}_{n-9,n}(n)|\leq 13$. Next we show that $|S^{10}_{n-9,n}(n)|\not=13$. 

\vskip 0.2cm
{\bf Claim 1} $|S^{10}_{n-9,n}(n)|\not=13$. 
\vskip 0.2cm
{\em Proof of Claim 1.} Seeking a contradiction, suppose that  $|S^{10}_{n-9,n}(n)|=13$. As we mentioned in the proof of Lemma \ref{maximofranja10}, 
the packing sets of $T^{10}_{n-9,n}(n)$ shown in Figure \ref{fig:unicosT10}
are the only sets with exactly 13 vertices.

Moreover, note that if $S^{10}_{n-9,n}(n)$ is any of (a), (b) or (c) in Figure \ref{fig:unicosT10}, then $|S^5_{n-10}(n)|=0$, which
 contradicts Lemma \ref{claim1} (ii). Hence we can assume that $S^{10}_{n-9,n}(n)$ is the packing set shown in Figure \ref{fig:unicosT10} (d). 
 
 Now consider the packing set $Q$ of $T^{n-9,n}_{n-13,n-5}(n)$ 
formed by the black, red and blue vertices in Figure \ref{casosT10con13:tiles}. Note that $S^{10}_{n-9,n-5}(n)$ is precisely the subset of 
$Q$ formed by the black and red vertices. 

\begin{figure}[H]
\begin{center}
\includegraphics[width=.85\textwidth]{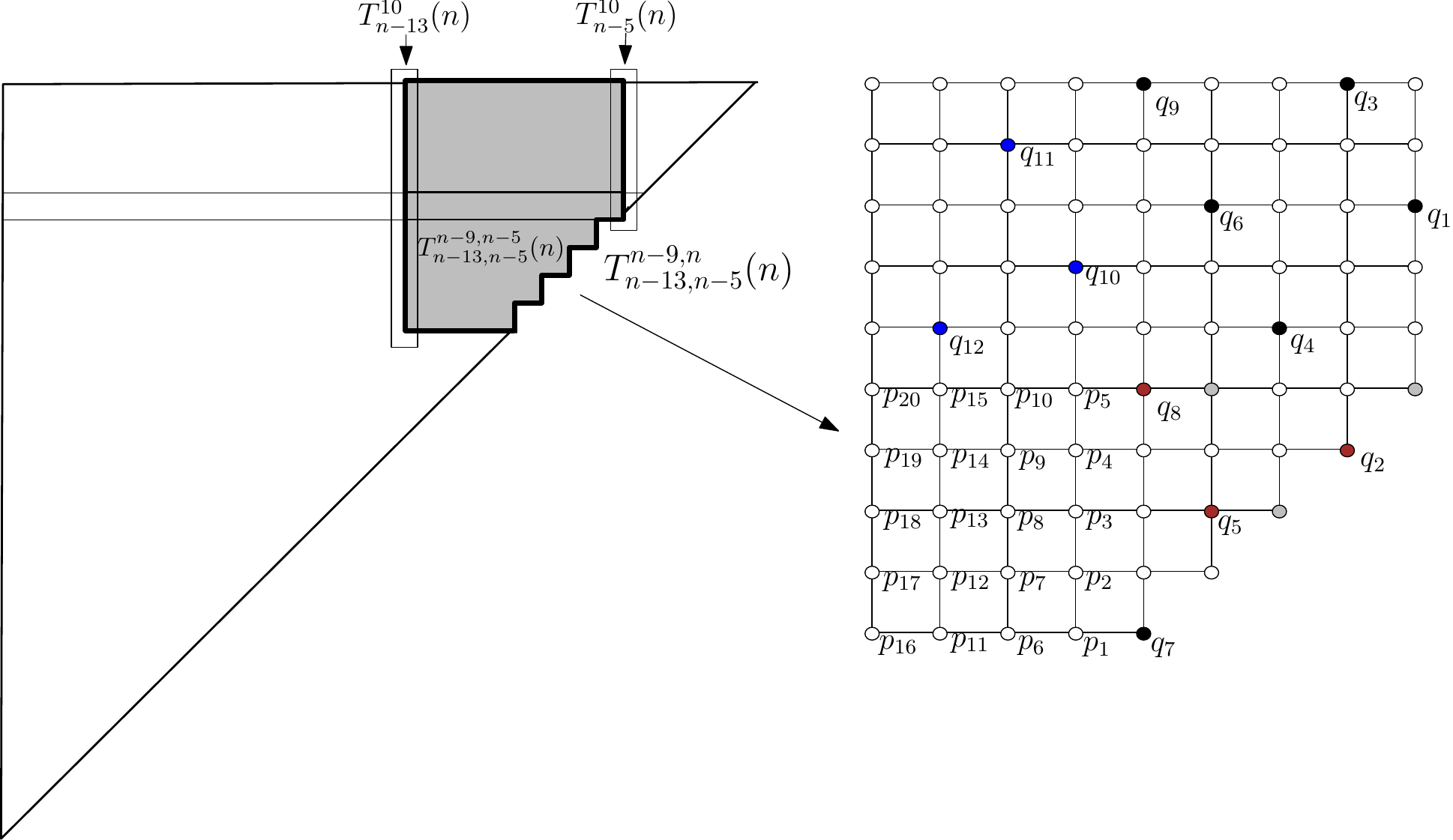}
\caption{\small This is $T^{n-9, n}_{n-13, n-5}(n)$. The set of black and brown vertices is $S^{10}_{n-9,n-5}(n)$.}
\label{casosT10con13:tiles}
\end{center}
\end{figure}

We label the points of $Q\cup T^{n-9,n-5}_{n-13,n-10}(n)$ with $p_1, p_2,\ldots ,p_{20}, q_1, q_2,\ldots ,q_{12}$ according to Figure \ref{casosT10con13:tiles}. Notice that $Q=\{q_1, q_2,\ldots , q_{12}\}$.

From Lemma \ref{claim1} (ii), we know that $|S^5_{n-12}(n)|=|S^5_{n-11}(n)|=|S^5_{n-10}(n)|=1$. These equalities and the locations of $q_6, q_8$ and $q_9$ imply that the only
 available location for the vertex in $S^5_{n-10}(n)$ is precisely the location of $q_{10}$, and so 
$S^5_{n-10}(n)=\{q_{10}\}$. Similarly, we can deduce that  $S^5_{n-11}(n)=\{q_{11}\}$ and $S^5_{n-12}(n)=\{q_{12}\}$.  
Since $q_5,q_7,q_8$ and $q_{12}$ are elements of $S$, then it is easy to verify that 
\begin{equation}\label{eq:pi}
\{p_1,p_2,p_3,p_4,p_5, p_{10},p_{14},p_{15}, p_{20}\}\cap S=\emptyset.          
\end{equation}

We finish the proof of claim by constructing a packing set $R$ of $T(n-5)$ with cardinality $|T(n-5)\cap S|+1$. Note that the existence of such an $R$ contradicts 
the assumption that $\rho(T(n-5))=|T(n-5)\cap S|$.  

First consider the set $R':= (S\setminus S^5(n))\cup \{p_{10}\}$ and let 
$N:=\{u\in S\setminus S^5(n) : d(u,p_{10})\leq 2\}$. Note that if $N=\emptyset$, then $R'$ is the required $R$. 

Now we show that if $N\not=\emptyset$, then we can always slightly modify $R'$ in order to get the required packing set $R$. Indeed, if 
neither $p_8$ nor $p_9$ belongs to $N$, then Equation (\ref{eq:pi}) implies that $N=\{q_8\}$.  In this case we simply interchange $q_2, q_5$ and $q_8$ by the three gray points 
in Figure \ref{casosT10con13:tiles} and the resulting set is the required $R$. 

On the other hand, suppose that at least one of $p_8$ or $p_9$ belongs to $N$. Because $S$ is a packing set, then exactly one of $p_8$ or $p_9$ belongs to $N$. Let $p_r$ be such an vertex. Since $p_r\in S$, then neither $p_7$ nor $p_{13}$ belongs to $R'$. From these facts and Equation (\ref{eq:pi}) it follows that  $N=\{p_r,q_8\}$. As before,  we interchange $p_r, q_2, q_5$ and $q_8$ by $p_3$ and the three gray points in Figure \ref{casosT10con13:tiles}. Clearly, the resulting set is the required $R$.  
$\triangle$

In view of Claim 1 and the fact $12\leq |S^{10}_{n-9,n}(n)|\leq 13$, we have that $|S^{10}_{n-9,n}(n)|=12$. 

 As $T^5_{n-4,n}(n) \simeq T(5)$ and $\rho (T(5))=4$,  then $|S^5_{n-4,n}(n)|\leq 4$.  If $|S^5_{n-4,n}(n)|<4$ then from Lemma \ref{claim1} (iii) it follows that
\[
|S^5(n)|=|S^5_{1,n-5}(n)|+|S^5_{n-4,n}(n)|<n-5+4=n-1,
\] which contradicts $|S^5(n)|=n-1$. Therefore we can assume that $|S^5_{n-4,n}(n)|=4$. From Lemma \ref{claim1} (ii) we know that $|S^5_{n-9,n-5}(n)|=5$ and hence 
\[
|S^{n-9,n-5}_{n-9,n-5}(n)|=|S^{10}_{n-9,n}(n)|-|S^5_{n-4,n}(n)|-|S^5_{n-9,n-5}(n)|=12- 4 - 5 = 3.
\] 
Next we show that $|S^{10}_{n-10}(n)|=2$ by contradiction. We have the following cases: 

\begin{itemize}
\item [1)]  $|S^{10}_{n-10}(n)|\leq 1$. From the fact that  $T^{10}_{1,n-11}(n)\simeq G_{10,n-11}$ and Equation~(\ref{eq:Gpq}) it follows that
$|S^{10}_{1,n-11}(n)|\leq \rho(G_{10,n-11})=2(n-11)$. This and $|S^{10}_{n-9,n}(n)|=12$ imply 
\[
|S^{10}(n)|=|S^{10}_{1,n-11}(n)|+|S^{10}_{n-10}(n)|+|S^{10}_{n-9,n}(n)|\leq 2(n-11)+1+12 = 2n-9, 
\]  a contradiction. 

\item [2)] $|S^{10}_{n-10}(n)|=3$. We have verified exhaustively by computer that there does not exist a packing set of $T^{n-9, n}_{n-14, n}(n)$ satisfying the current properties of 
$S$. We recall that such properties are the following:
 
\begin{itemize}
\item[(P1)] $|S^{10}_{n-9,n}(n)|=12$,
\item[(P2)]   $S^5_i(n)=1$ for every $i\in \{n-14,n-13,\ldots, n-5\}$, (by Lemma \ref{claim1} (ii))
\item[(P3)]  $|S^5_{n-4, n}(n)|=4$,
\item[(P4)]  $|S^{n-9, n-5}_{n-9, n-5}(n)|=3$, 
\item[(P5)]  $|S^{n-9,n-5}_{n-10}(n)|=2$, (because $|S^{10}_{n-10}(n)|=3$ and $|S^5_{n-10}|=1$ by Lemma \ref{claim1} (ii)),
\item[(P6)]  $|S^{n-9, n-5}_{i}(n)|=1$ for  every $i\in \{n-14,n-13, n-12\}$ (by Proposition \ref{proposicionunounofranjas}).
\end{itemize}

We remark that (P2), (P3), (P4) and (P6) are consequences of (P1) and hold independently on whether or not $|S^{10}_{n-10}(n)|=3$.  

\item [3)] $|S^{10}_{n-10}(n)|=4$. From Lemma \ref{claim1} (ii) we know that $|S^5_{n-10}(n)|=1$, and so  $|S^{n-9,n-5}_{n-10}(n)|=3$, which is impossible. 
\end{itemize}
\begin{figure}[H]
\begin{center}
\includegraphics[width=.89\textwidth]{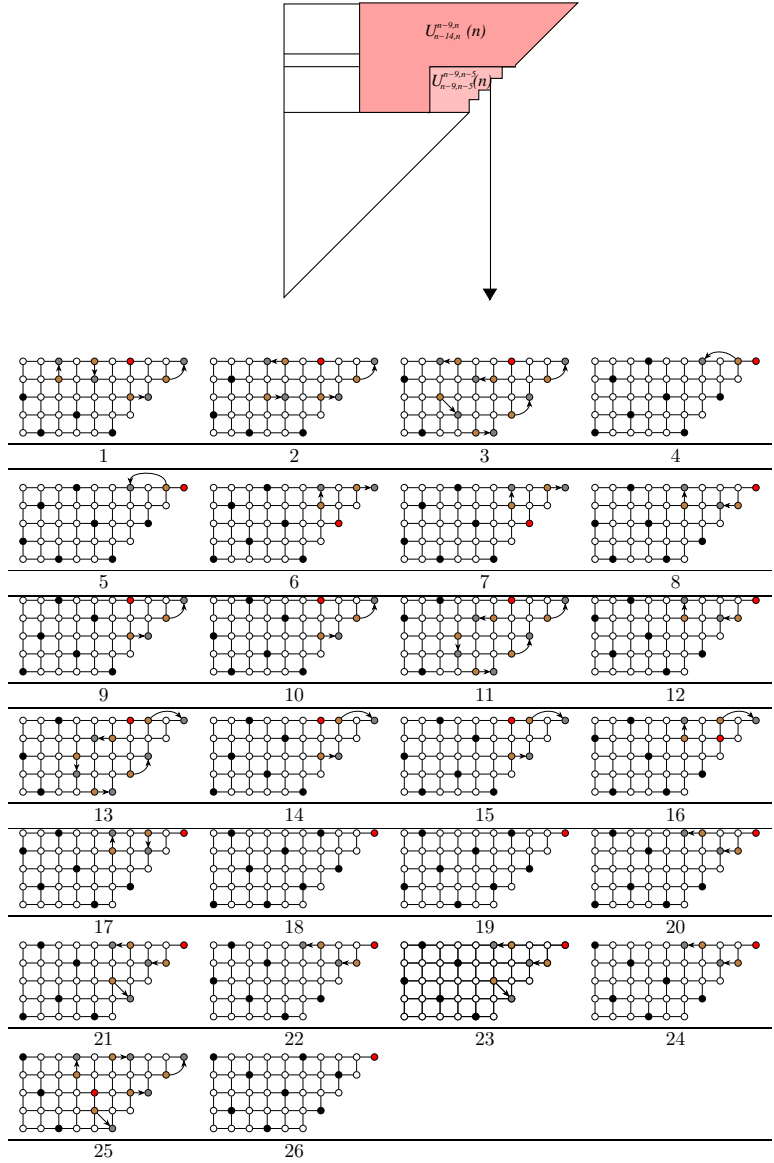}
\caption{\small In each of these 26 cases, the union of the black and brown vertices is a packing set of $T_{n-14,n-5}^{n-9,n-5}(n)$ which satisfies (P1), (P2), (P3), (P4), (P6) and additionally $|S^{n-9,n-5}_{n-11}(n)|=1$ and $|S^{n-9,n-5}_{n-10}(n)|=1$. Moreover, these 26 packing sets are the only satisfying all these conditions. Finally, note that in each case,  the set $R$ that results when we interchange the brown vertices by the red and yellow vertices in $S\setminus S^5(n)$ is a packing set of $T(n-5)$ with cardinality 
$|S\setminus S^5(n)|+1$.}
\label{fig:muestra}
\end{center}
\end{figure}

Therefore we have that $|S^{10}_{n-10}(n)|=2$. This and Lemma \ref{claim1} (ii) imply $|S^{n-9,n-5}_{n-10}(n)|=1$. Thus $|S^{10}_{n-10,n}(n)|=14$.
 Now we will prove that $|S^{n-9,n-5}_{n-11}(n)|=1$. By Proposition ~\ref{proposicionunounofranjas} we know that $|S^{n-9,n-5}_{n-13}(n)|=|S^{n-9,n-5}_{n-12}(n)|=1$. These equalities and Proposition~\ref{obsno32} imply $|S^{n-9,n-5}_{n-11}(n)|\leq 1$.  
 If $|S^{n-9,n-5}_{n-11}(n)|=0$, then Lemma \ref{claim1} (ii) imply that $|S^{10}_{n-11}(n)|=|S^5_{n-11}(n)|+|S^{n-9,n-5}_{n-11}(n)|=1$. From this and Equation (1) we have that 
\[
 |S^{10}(n)|=|S^{10}_{1,n-12}(n)|+|S^{10}_{n-11}(n)|+|S^{10}_{n-10,n}(n)|=2(n-12)+1+14=2n-9,
\]
which contradicts $|S^{10}(n)|=2n-8$. Hence we have that $|S^{n-9,n-5}_{n-11}(n)|=1$.

We have obtained exhaustively by computer all the possibilities for $S^{10}_{n-14,n}(n)$ which satisfy (P1), (P2), (P3), (P4), (P6) and additionally $|S^{n-9,n-5}_{n-11}(n)|=1$ and
$|S^{n-9,n-5}_{n-10}(n)|=1$. These  possibilities restricted to $S^{n-9,n-5}_{n-14,n-5}(n)$ are exactly $26$ and they are shown in Figure~\ref{fig:muestra}. 

The points of $S^{n-9,n-5}_{n-14,n-5}(n)$ are coloured black and brown in Figure \ref{fig:muestra}. For each of these 26 possibilities, we proceed similarly as in the proof of 
Claim 1 in order to exhibit a packing set $R$ of $T(n-5) $ with cardinality $|T(n-5)\cap S|+1$. Again, this contradicts $\rho(T(n-5))=|T(n-5)\cap S|$ and will allow us to conclude that $|S^{10}_{n-9,n}(n)|\not=12$, and hence the proof of Lemma  \ref{lemma:upper}.
Now we show that for each of these 26 possibilities, such a packing set $R$ exists. Indeed, in each of these cases take $R$ as the union of $S\setminus S^5(n)$ with the red, black and yellow vertices. In each case it is easy to verify that such an set $R$ is the required packing set of $T(n-5)$.
$\square$

The following corollary prove the conjecture of Rob Pratt about  sequence A085680. 
\begin{cor} 
$\sum_{n\geq 0} A085680(n+2)x^n=\frac{1-x+x^2-x^{10}+x^{11}}{(1-x)^2(1-x^5)}$
\end{cor}

\begin{proof}
As $a(n)=A085680(n+1)$ for $n\geq 6$, we can use the recursion $a(n)=a(n-5)+n-2$ to obtain $A085680(n)=A085680(n-5)+(n-3)$ for $n\geq 12$, with initial values in Table~\ref{tablaan}, and then, the desired generating function can be deduced by standard techniques. 
\end{proof}
  
\section*{Acknowledgements}

The authors would like to thank Neil Sloane for useful email exchanges and for his  explanation about the graph theory version of the error correcting code problem counted by sequence A085680. 


\end{document}